\documentclass[11pt,dvipsnames]{amsart}

\usepackage[english]{babel}
\usepackage[utf8]{inputenc}
\usepackage{amssymb,amsthm,amsfonts,amsmath}
\usepackage{mathrsfs}
\usepackage{verbatim}
\usepackage{mathtools}
\usepackage[shortlabels]{enumitem}
\usepackage[margin=1in]{geometry}

\usepackage{tikz}
\usepackage[backref=page,colorlinks,breaklinks,linkcolor=BrickRed,citecolor=OliveGreen,urlcolor=blue,hypertexnames=true]{hyperref}

\newcommand{\PhiAff}{\mathcal{P}_{\mathrm{biaff}}}
\newcommand{\CAff}{\mathcal{C}_{\mathrm{aff}}}

\newtheorem{theorem}{Theorem}[section]
\newtheorem{corollary}[theorem]{Corollary}
\newtheorem{lemma}[theorem]{Lemma}
\newtheorem{proposition}[theorem]{Proposition}

\theoremstyle{definition}
\newtheorem{definition}[theorem]{Definition}
\newtheorem{remark}[theorem]{Remark}
\newtheorem{example}[theorem]{Example}

\newcommand{\R}{\mathbb{R}}

\begin{document}
	
	\title{Kadison duality for biaffinely convex sets}

	\author[T. \v{S}trekelj]{Tea \v{S}trekelj}
	\address{Institute of Mathematics, University of Innsbruck, Austria.}
	\email{Tea.Strekelj@uibk.ac.at}

	\thanks{The author was supported by the Slovenian Research and Innovation Agency grant J1-60011.}

	\subjclass[2020]{46A55, 52A01, 46A40, 46L08, 46E15}
	
	\keywords{convex set, affine function, Archimedean order unit space, state space, Kadison duality, biconvex set, biaffinely convex set, biaffine polynomial, biaffine order unit space}

	\begin{abstract}
		Classical Kadison duality establishes a categorical equivalence between compact convex sets and Archimedean order unit spaces of affine functions. This paper develops a Kadison-type duality for sets with coordinate-wise convex structure in product domains $\mathbb{R}^n \times \mathbb{R}^m$.  While slice-wise convex (biconvex) sets appear to be the natural geometric domain, separate convexity along coordinates is insufficient to ensure separation by biaffine polynomials.
		In classical convexity, every closed convex set is determined by affine functionals because any exterior point can be separated by a hyperplane. In contrast, biaffine polynomials can only probe a biconvex set $K$ up to its biaffine hull $\operatorname{biconv}(K)$, the envelope formed by the intersection of all regions cut out by valid biaffine inequalities on $K$. 
		To resolve this obstruction, the geometric setting is restricted to biaffinely convex sets, those that coincide with their biaffine hull ($K = \operatorname{biconv}(K)$) and therefore admit strict biaffine separation. On the algebraic side, we introduce the dual objects, which we call  biaffine order unit spaces. These are finite-dimensional Archimedean order unit spaces defined on tensor products $E_1 \otimes E_2$ whose order units factor and whose positive cones are governed entirely by unentangled, rank-1  tensor evaluations. 
		We prove a duality between the category of compact non-degenerate biaffinely convex sets with coordinate-wise affine maps and the category of abstract biaffine order unit spaces with tensor-factorized unital positive linear maps. In particular, every such biaffinely convex set is isomorphic to the coordinate state space of its biaffine polynomials, and conversely, every biaffine order unit space is order-isomorphic to the space of biaffine polynomials on its coordinate state space. 
	\end{abstract}
	
	\maketitle
	
	\newpage
	
	\section{Introduction}
	
	Convexity is a foundational geometric concept that underpins vast areas of mathematics, from optimization and linear programming to functional analysis and probability theory. A subset of a vector space is convex if it contains the line segment connecting any two of its points, or equivalently, if it is closed under convex combinations of its points. In modern analysis, the structural power of convexity lies predominantly in its dual relationship with affine functions defined upon the set. The cornerstone of this perspective is the classical Kadison duality theorem \cite{Al}, which establishes a canonical equivalence between compact convex sets and function systems. The latter are Archimedean order unit spaces, concretely realized as spaces of continuous affine functions. A compact convex set $K$ is completely characterized by the order-theoretic and metric properties of its canonical function system, while conversely, any abstract function system is geometrically recovered through its state space of normalized positive linear functionals. 
	
	Building upon this classical foundation, several generalized convexity notions have been developed to capture richer algebraic and geometric behaviors \cite{MM96, Nai04}. In the noncommutative, matrix convexity plays a central role \cite{Wit, EW, WW, HM12, HM14, HKM17, zalar,  HL, DK}. There, scalar sets are replaced by graded families of matrix tuples closed under direct sums and compressions. In direct analogy to Kadison duality, Webster–Winkler duality demonstrates that matrix convex sets are dual to operator systems \cite{WW}, providing indispensable geometric tools for operator spaces, $C^*$-algebras, and noncommutative function theory \cite{Pa, DM05, Arv08, DK15}. More recently, these notions have been extended to coordinate-dependent settings through $\Gamma$-convexity, including noncommutative partial convexity \cite{HHLM08, DHM, JKMMP21, JKMMP22, KMS26}, where sets exhibit convexity only with respect to designated coordinate blocks.
	
	In the commutative setting, Kadison duality was recently extended to partially convex sets, establishing a correspondence with free order unit modules over continuous function algebras \cite{Str26}.  A natural progression is to consider sets equipped with coordinate-wise convex structure simultaneously across multiple variables, specifically, subsets $K \subseteq \mathbb{R}^n \times \mathbb{R}^m$ whose directional slices are convex in both $x \in \R^n$ and $y \in \R^m$. While biconvex sets appear to be the primary geometric candidate and biaffine polynomials the corresponding function system, a significant geometric obstruction arises: separate slice-convexity is inherently too coarse to guarantee separation by biaffine polynomials. As a result, a biconvex set can be strictly smaller than its biaffine hull $\operatorname{biconv}(K)$, the envelope formed by the intersection of all regions cut out by valid biaffine inequalities on $K$.  
	
	The aim of this paper is to resolve this obstruction and formulate a Kadison-type duality for coordinate-wise convex geometry. By restricting to biaffinely convex sets, those that coincide with their biaffine hull, we establish a duality with finite-dimensional, tensor-factorized Archimedean order unit spaces, which we introduce as biaffine order unit spaces. 
	
	\subsection*{Use of AI} The author used Google’s Gemini for assistance with exposition and with structuring and streamlining mathematical arguments. The main ideas and proof strategies were conceived by the author, who independently checked and verified all mathematical steps and takes full responsibility for the content.
	
	\subsection{Set definitions and ambient embedding}
	
	Throughout, we consider subsets of $\R^{n+m}$ whose points will be denoted by $(x,y),$ where $x \in \R^n$ and $y\in \R^m.$

	\begin{definition}
		Let $V = \mathbb{R}^{n+m} = \mathbb{R}^n \times \mathbb{R}^m$. For a subset $K \subseteq V$, let 
		\[
		\pi_1(K) = \{x \in \mathbb{R}^n \mid \exists y \in \mathbb{R}^m : (x,y) \in K\} \quad \text{and} \quad \pi_2(K) = \{y \in \mathbb{R}^m \mid \exists x \in \mathbb{R}^n : (x,y) \in K\}
		\]
		denote the canonical projections onto $\mathbb{R}^n$ and $\mathbb{R}^m$, respectively. For each fixed $y_0 \in \pi_2(K)$ and $x_0 \in \pi_1(K)$, denote the directional slices by
		\[
		K_{y_0} = \{x \in \mathbb{R}^n \mid (x, y_0) \in K\} \quad \text{and} \quad K_{x_0} = \{y \in \mathbb{R}^m \mid (x_0, y) \in K\}.
		\]
		The set $K \subseteq \mathbb{R}^{n+m}$ is a \textbf{biconvex set} (or partially convex in both $x$ and $y$) if
		\begin{enumerate}[(i)]
			\item for every $y_0 \in \pi_2(K)$, the slice $K_{y_0} \subseteq \mathbb{R}^n$ is a convex set;
			\item for every $x_0 \in \pi_1(K)$, the slice $K_{x_0} \subseteq \mathbb{R}^m$ is a convex set.
		\end{enumerate}
	\end{definition}
	
	In classical convex analysis, the primary tools for probing convex sets are affine functions. For biconvex sets, the natural analogues are functions that behave affinely along each directional slice. 
	
	\begin{definition}
		The space of \textbf{biaffine polynomials}, denoted by $\mathcal{P}_{\mathrm{biaff}}(\R^{n+m})$, consists of all polynomials $p(x,y)$ that are affine in $x$ for every fixed $y$, and affine in $y$ for every fixed $x$. Any such polynomial can be uniquely represented in the form
		\[
		p(x,y) = c_{00} + \sum_{i=1}^n c_{i0}x_i + \sum_{j=1}^m c_{0j}y_j + \sum_{i=1}^n \sum_{j=1}^m c_{ij}x_i y_j,
		\]
		where $c_{00}, c_{i0}, c_{0j}, c_{ij} \in \mathbb{R}$.
		For any $K \subseteq \mathbb{R}^{n+m},$ we set 
		$$
		\mathcal{P}_{\mathrm{biaff}}(K) = \{p|_K \mid p \in \mathcal{P}_{\mathrm{biaff}}(\R^{n+m})  \}.
		$$
\end{definition}

Just as a closed convex set can be characterized as an intersection of closed affine half-spaces, one can consider sets cut out by families of biaffine inequalities.

\begin{definition}\label{def:bihull}
	For a subset $K \subseteq \mathbb{R}^n \times \mathbb{R}^m$, the \textbf{biaffine hull} (or the set defined by all valid biaffine inequalities on $K$) is defined by
	\[
	\mathrm{biconv}(K) := \bigcap \left\{ \{(x,y) \in \mathbb{R}^{n+m} \mid p(x,y) \ge 0\} \;\middle|\; p \in \mathcal{P}_{\mathrm{biaff}}(\mathbb{R}^{n+m}), \; p|_K \ge 0 \right\}.
	\]
\end{definition}

Although biaffine polynomials contain bilinear cross-terms $x_i y_j$, they can be ``linearized'' by mapping the ambient space $\mathbb{R}^{n+m}$ into a higher-dimensional tensor product space via the Segre embedding. 

\begin{definition}
	Let $D = (n+1)(m+1)$. The canonical biaffine embedding $\Phi: \mathbb{R}^{n+m} \to \mathbb{R}^D$ is defined by
	\[
	\Phi(x,y) := \begin{pmatrix} 1 \\ x \end{pmatrix} \otimes \begin{pmatrix} 1 \\ y \end{pmatrix} = \left(1, (x_i)_{i=1}^n, (y_j)_{j=1}^m, (x_i y_j)_{i=1,j=1}^{n,m}\right)^T.
	\]
	Denote $\Sigma_{n,m} := \Phi(\mathbb{R}^{n+m}) \subseteq \mathbb{R}^D.$ 
\end{definition}

The embedding $\Phi$ transforms biaffine inequalities on $\mathbb{R}^{n+m}$ into linear hyperplanes in $\mathbb{R}^D$. This geometric lifting enables us to characterize sets that coincide with their biaffine hull in terms of classical convex hulls in the ambient tensor space.

\begin{lemma}\label{le:phieq}
	Let $K \subseteq \mathbb{R}^{n+m}$ be a compact set.
	\begin{enumerate}
		\item[(a)] The biaffine hull of $K$ is given by
		\[
		\operatorname{biconv}(K) = \Phi^{-1}(\operatorname{conv}(\Phi(K))).
		\]
		\item[(b)] The following conditions are equivalent:
		\begin{enumerate}
			\item[(1)] $K = \operatorname{biconv}(K)$;
			\item[(2)] $\Phi(K) = \operatorname{conv}(\Phi(K)) \cap \Sigma_{n,m}$;
			\item[(3)] $\Phi(z) \notin \operatorname{conv}(\Phi(K))$ for every $z \in \mathbb{R}^{n+m} \setminus K$.
		\end{enumerate}
	\end{enumerate}
\end{lemma}

\begin{proof}
	(a) By Definition~\ref{def:bihull}, $z \in \operatorname{biconv}(K)$ if and only if every biaffine polynomial non-negative on $K$ is non-negative at $z$. Every such polynomial is of the form $p(w) = \langle c, \Phi(w) \rangle - \gamma$ for some $c \in \mathbb{R}^D$ and $\gamma \in \mathbb{R}$. Therefore,
	\begin{align*}
		z \notin \operatorname{biconv}(K) 
		&\iff \exists c \in \mathbb{R}^D, \gamma \in \mathbb{R} \text{ such that } \langle c, \Phi(w) \rangle \ge \gamma \; (\forall w \in K) \text{ and } \langle c, \Phi(z) \rangle < \gamma \\
		&\iff \exists c \in \mathbb{R}^D, \gamma \in \mathbb{R} \text{ such that } \langle c, W \rangle \ge \gamma \; (\forall W \in \operatorname{conv}(\Phi(K))) \text{ and } \langle c, \Phi(z) \rangle < \gamma \\
		&\iff \Phi(z) \notin \operatorname{conv}(\Phi(K)),
	\end{align*}
	where the final equivalence is the strict Hahn-Banach separation theorem applied to the compact convex set $\operatorname{conv}(\Phi(K)) \subseteq \mathbb{R}^D$ and the point $\Phi(z)$. Taking the contrapositive yields $z \in \operatorname{biconv}(K) \iff \Phi(z) \in \operatorname{conv}(\Phi(K))$, which establishes (a).
	
	(b) Applying $\Phi$ to both sides of the identity in (a) and intersecting with the image $\Sigma_{n,m} = \Phi(\mathbb{R}^{n+m})$ gives
	\[
	\Phi(\operatorname{biconv}(K)) = \operatorname{conv}(\Phi(K)) \cap \Sigma_{n,m}.
	\]
	Since $\Phi: \mathbb{R}^{n+m} \to \Sigma_{n,m}$ is a bijection, $K = \operatorname{biconv}(K)$ holds if and only if $\Phi(K) = \Phi(\operatorname{biconv}(K)) = \operatorname{conv}(\Phi(K)) \cap \Sigma_{n,m}$, proving $(1) \Leftrightarrow (2)$. The equivalence $(2) \Leftrightarrow (3)$ is immediate since $\Sigma_{n,m} \setminus \Phi(K) = \Phi(\mathbb{R}^{n+m} \setminus K)$.
\end{proof}


\begin{definition}
	\label{def:biaff_convex}
	A compact subset $K \subseteq \mathbb{R}^{n+m}$ is called \textbf{biaffinely convex} if it satisfies any of the following equivalent conditions:
	\begin{enumerate}
		\item $K$ is equal to its biaffine hull:
		\[
		K = \operatorname{biconv}(K) := \bigcap \left\{ \{w \in \mathbb{R}^{n+m} \mid p(w) \ge 0\} \;\middle|\; p \in \mathcal{P}_{\mathrm{biaff}}(\mathbb{R}^{n+m}), \; p|_K \ge 0 \right\}.
		\]
		\item The lifted set $\Phi(K)$ satisfies:
		\[
		\Phi(K) = \operatorname{conv}(\Phi(K)) \cap \Sigma_{n,m}.
		\]
		Equivalently, for every $z \in \mathbb{R}^{n+m} \setminus K$, $\Phi(z) \notin \operatorname{conv}(\Phi(K))$.
	\end{enumerate}
\end{definition}


\begin{lemma}\label{le:sub}
	Let $K \subseteq \mathbb{R}^{n+m}$ be a biaffinely convex set. Then $K$ is a biconvex set.
\end{lemma}

\begin{proof}
	Since $K$ is biaffinely convex, it coincides with its biaffine hull
	\[
	K = \operatorname{biconv}(K) = \bigcap_{p \in \mathcal{F}} H_p,
	\]
	where $\mathcal{F} = \{p \in \mathcal{P}_{\mathrm{biaff}}(\mathbb{R}^{n+m}) \mid p|_K \ge 0\}$ and $H_p = \{(x,y) \in \mathbb{R}^{n+m} \mid p(x,y) \ge 0\}$.
	
	To establish biconvexity, we verify the convexity of both directional slices.
	First, fix $y_0 \in \pi_2(K)$. The vertical slice of $K$ at $y_0$ is given by
	\[
	K_{y_0} = \{x \in \mathbb{R}^n \mid (x, y_0) \in K\} = \bigcap_{p \in \mathcal{F}} \{x \in \mathbb{R}^n \mid p(x, y_0) \ge 0\}.
	\]
	By definition of $\mathcal{P}_{\mathrm{biaff}}(\mathbb{R}^{n+m})$, for each fixed $y_0$, the restriction $x \mapsto p(x, y_0)$ is an affine polynomial on $\mathbb{R}^n$. Consequently, each set
	\[
	(H_p)_{y_0} = \{x \in \mathbb{R}^n \mid p(x, y_0) \ge 0\}
	\]
	is an affine half-space (or the entire space $\mathbb{R}^n$, or $\emptyset$), which is convex. Since the arbitrary intersection of convex sets is convex, the slice $K_{y_0} = \bigcap_{p \in \mathcal{F}} (H_p)_{y_0}\subseteq \mathbb{R}^n$ is convex  for every $y_0 \in \pi_2(K)$. Similarly, the slice $K_{x_0} \subseteq \R^m$ is convex for all $x_0 \in \pi_1(K).$
\end{proof}

While Lemma \ref{le:sub} shows that every biaffinely convex set is biconvex, the converse generally fails, as we will demonstrate in Section \ref{sec:2}. Before exploring this obstruction, we introduce the geometric morphisms and non-degeneracy conditions needed to formulate our duality.

\begin{definition} \label{def:biiso}
	Two biaffinely convex sets $K, L \subseteq \mathbb{R}^{n+m}$ are isomorphic if there exist affine isomorphisms $\phi: \mathbb{R}^n \to \mathbb{R}^n$ and $\psi: \mathbb{R}^m \to \mathbb{R}^m$ such that the product map
	$$\Theta: \mathbb{R}^{n+m} \longrightarrow \mathbb{R}^{n+m}, \quad \Theta(x, y) := (\phi(x), \psi(y))$$
	restricts to a bijection between $K$ and $L$, i.e., $\Theta(K) = L$.
\end{definition}

\begin{definition}\label{def:nondeg}
	A subset $K \subseteq \mathbb{R}^{n+m}$ is called \textbf{non-degenerate of type $(n,m)$} if the lifted set $\Phi(K)$ spans the full space,
	$$\operatorname{span}(\Phi(K)) = \mathbb{R}^{(n+1)(m+1)}.$$
\end{definition}

\begin{remark}\label{rem:non-deg}
	Note that if $K \subseteq \mathbb{R}^{n+m}$ is non-degenerate, then $\pi_1(K)$ affinely spans $\mathbb{R}^n$ (ensuring $\dim(\operatorname{Aff}(\pi_1(K))) = n+1$), and similarly, $\pi_2(K)$ affinely spans $\mathbb{R}^m$ (ensuring $\dim(\operatorname{Aff}(\pi_2(K))) = m+1$). Indeed,  if  $\pi_1(K)$ failed to affinely span $\R^n$, a nonzero
	affine function of $x$ would vanish on $K$, giving a nonzero vector annihilating every $\Phi(x,y)$ for $(x,y)\in K$.
\end{remark}


\begin{lemma}\label{le:inj}
	If $K \subseteq \mathbb{R}^{n+m}$ is non-degenerate, the restriction map
	$$\mathcal{P}_{\textnormal{biaff}}(\mathbb{R}^{n+m}) \to \mathcal{C}(K), \quad p \mapsto p\vert{}_K$$
	is injective. In particular, every $p \in \mathcal{P}_{\textnormal{biaff}}(K)$ has a unique coefficient vector $c \in \mathbb{R}^{(n+1)(m+1)}$ such that $p(x,y) = \langle c, \Phi(x,y)\rangle$.  
\end{lemma}

\begin{proof}
	If $p\vert{}_K \equiv 0$, then $\langle c, \Phi(w)\rangle = 0$ for all $w \in K$. Because $\operatorname{span}(\Phi(K)) = \mathbb{R}^{(n+1)(m+1)}$, $c$ is orthogonal to the entire ambient space, forcing $c = 0$.  
\end{proof}

\subsection{Paper overview and main result}
This paper establishes a Kadison-type duality for biaffinely convex sets. In classical Kadison duality, compact convex sets are categorically equivalent to Archimedean order unit spaces of continuous affine functions \cite{Al}. While recent generalizations to partially convex sets model one-sided convexity through dual function systems that are free modules over commutative $C^*$-algebras \cite{Str26}, restricting both coordinate directions to affine dependencies simplifies the dual geometry. In fact, the dual object transitions from an infinite-dimensional module bundle to a finite-dimensional, tensor-factorized Archimedean order unit (AOU) space of dimension $(n+1)(m+1)$.

The primary geometric candidate for this two-sided affine structure is the class of biconvex sets, whose natural dual function space is $\mathcal{P}_{\text{biaff}}(K)$. However, separate slice-convexity is intrinsically insufficient to guarantee separation by biaffine polynomials. Because biaffine polynomials linearize via the non-linear embedding $\Phi$, separating an exterior point $z \notin K$ requires that $\Phi(z) \notin \operatorname{conv}(\Phi(K))$. However, taking convex combinations in the higher-dimensional ambient space mixes points across different coordinate slices. As demonstrated in Section \ref{sec:2}, an exterior point $z \notin K$ can have its lifted image trapped inside $\operatorname{conv}(\Phi(K))$, preventing biaffine separation and resulting in $K \subsetneq \operatorname{biconv}(K)$.


Because biaffine polynomials can only characterize sets cut out by their own inequalities, the geometric objects are restricted to biaffinely convex sets ($K = \operatorname{biconv}(K)$). On this class, strict separation holds by construction. 
In Section \ref{sec:3}, we introduce the abstract dual objects, which we call  biaffine order unit spaces. These are finite-dimensional AOU spaces defined on tensor products $E_1 \otimes E_2$ whose positive cones are governed entirely by unentangled rank-1 product states. 

Conversely, to recover the geometric object from an abstract biaffine order unit space $A \cong E_1 \otimes E_2$, one cannot simply take its full state space $S(A)$. In fact, $S(A)$ is convex and corresponds to an ambient convex hull containing higher-rank (entangled) states. Reconstructing the set requires isolating the unentangled rank-1 product states, the biaffine state space $S_{\mathrm{biaff}}(A)$, which represent genuine point evaluations. Choosing affine coordinate bases for $E_1$ and $E_2$ then embeds $S_{\mathrm{biaff}}(A)$ into $\mathbb{R}^{n+m}$ as the coordinate biaffine state space $K_A$, uniquely determined up to product-affine isomorphism from Definition \ref{def:biiso}. Our main result proves that compact non-degenerate biaffinely convex sets and biaffine order unit spaces are dual to each other:   

\begin{theorem}\phantom{}\\ \vspace{-4mm}\label{thm:dual}
	\begin{enumerate}[(a)] 
		\item Let $K \subseteq \R^{n+m}$ be a compact non-degenerate biaffinely convex set. Then $\PhiAff(K)$ is a biaffine order unit space of type $(n,m)$ and $K$ is isomorphic to $K_{\PhiAff(K)}.$ 
		\item 	Let $A = E_1 \otimes E_2$ be an abstract biaffine order unit space. For any choice of coordinate bases for $E_1$ and $E_2$, the associated coordinate biaffine state space $K_A \subseteq \mathbb{R}^{n+m}$ is a compact non-degenerate biaffinely convex set, and $A$ is isomorphic to $\mathcal{P}_{\mathrm{biaff}}(K_A)$.
	\end{enumerate}
\end{theorem}

Finally, in Section \ref{sec:4}, we elevate this object-wise correspondence to a contravariant categorical equivalence. We introduce the category $\mathbf{BiAffConv}$ of compact non-degenerate biaffinely convex sets with coordinate-wise affine morphisms, 
and the category $\mathbf{BiAOU}$ of biaffine order unit spaces whose morphisms are factorized unital positive linear maps. 
In Theorem~\ref{thm:categorical_duality}, we prove that $\mathbf{BiAffConv}^{\mathrm{op}}$ is categorically equivalent to $\mathbf{BiAOU}$.

\section{The obstruction: separation failure of biaffine polynomials}\label{sec:2}
Central to convex analysis is the separation of a point from a closed convex set via affine hyperplanes. This separation property ensures that every closed convex set is an intersection of closed half-spaces; its geometry is thus entirely determined by affine functionals, establishing the space of all affine functionals as the canonical dual object.
In this section, we study the separation capabilities of biaffine polynomials and exhibit an explicit counterexample demonstrating that biconvexity alone does not guarantee biaffine separation.

\subsection{The natural candidate and the separation criterion}
In analogy to classical convexity (affine functions) and partial convexity (partially affine functions), $\mathcal{P}_{\mathrm{biaff}}(K)$ seems to be the natural dual candidate for a biconvex set $K$. To analyze the separation properties of $\mathcal{P}_{\text{biaff}}(K)$, we first reformulate biaffine polynomials in terms of the linear geometry of the ambient tensor space $\mathbb{R}^D$ via the embedding $\Phi$.   

\begin{lemma}
	\label{lem:linearization}
	For every biaffine polynomial $p \in \mathcal{P}_{\mathrm{biaff}}(\mathbb{R}^{n+m})$, there exists a unique vector $c \in \mathbb{R}^D$ such that
	\[
	p(x,y) = \langle c, \Phi(x,y) \rangle \quad \text{for all } (x,y) \in \mathbb{R}^{n+m}.
	\]
	Conversely, for every vector $c \in \mathbb{R}^D$, the function $(x,y) \mapsto \langle c, \Phi(x,y) \rangle$ is a biaffine polynomial in $\mathcal{P}_{\mathrm{biaff}}(\mathbb{R}^{n+m})$.
\end{lemma}

\begin{proof}
	Let $c = (c_{00}, (c_{i0})_{i=1}^n, (c_{0j})_{j=1}^m, (c_{ij})_{i=1,j=1}^{n,m})^T \in \mathbb{R}^D$. Taking the standard Euclidean inner product with $\Phi(x,y)$ gives:
	\[
	\langle c, \Phi(x,y) \rangle = c_{00} + \sum_{i=1}^n c_{i0} x_i + \sum_{j=1}^m c_{0j} y_j + \sum_{i=1}^n \sum_{j=1}^m c_{ij} x_i y_j = p(x,y).
	\]
	Linear independence of the standard monomial basis $\{1, x_i, y_j, x_i y_j\}$ guarantees uniqueness of the coefficient vector $c$.
\end{proof}

By representing each biaffine polynomial as a linear functional $\langle c, \Phi(\cdot) \rangle$, the problem of separating an exterior point $z \in \mathbb{R}^{n+m} \setminus K$ using a biaffine polynomial reduces to separating the lifted point $\Phi(z)$ from the lifted set $\Phi(K)$ via an affine hyperplane in $\mathbb{R}^D$.

\begin{proposition}
	Let $K \subseteq \mathbb{R}^{n+m}$ be a compact set, and let $z \in \mathbb{R}^{n+m} \setminus K$. If $\Phi(z) \notin \operatorname{conv}(\Phi(K))$, then there exists a biaffine polynomial $p \in \mathcal{P}_{\mathrm{biaff}}(\mathbb{R}^{n+m})$ such that
	\[
	p(w) \ge 0 \quad \text{for all } w \in K, \qquad \text{and} \qquad p(z) < 0.
	\]
\end{proposition}

\begin{proof}
	Since $K$ is compact and $\Phi$ is continuous, the lifted image $\Phi(K) \subset \mathbb{R}^D$ is compact, and hence its convex hull $\mathcal{C} = \operatorname{conv}(\Phi(K))$ is a compact convex subset of $\mathbb{R}^D$.
	
	Because $\Phi(z) \notin \mathcal{C}$, the strict Hahn--Banach separation theorem guarantees the existence of a vector $c \in \mathbb{R}^D$ and a scalar $\gamma \in \mathbb{R}$ such that:
	\[
	\langle c, W \rangle \ge \gamma \quad \text{for all } W \in \mathcal{C}, \qquad \text{and} \qquad \langle c, \Phi(z) \rangle < \gamma.
	\]
	Define $p(x,y) := \langle c, \Phi(x,y) \rangle - \gamma$. By Lemma~\ref{lem:linearization}, $p \in \mathcal{P}_{\mathrm{biaff}}(\mathbb{R}^{n+m})$. Since $\Phi(w) \in \mathcal{C}$ for all $w \in K$, we have $p(w) \ge 0$ on $K$, while $p(z) < 0$.
\end{proof}


\begin{corollary}\label{cor:sep} Let $K \subseteq \mathbb{R}^{n+m}$ be a compact set.
	A point $z \notin K$ is separable by a biaffine polynomial if and only if $\Phi(z) \notin \mathrm{conv}(\Phi(K))$. 
\end{corollary}

\begin{proof}
	Since every biaffine polynomial on $\mathbb{R}^{n+m}$ is of the form $w \mapsto \ell(\Phi(w))$ for some affine functional $\ell$ on $\mathbb{R}^D$ (Lemma~\ref{lem:linearization}), separating $z$ from $K$ by a biaffine polynomial is equivalent to separating $\Phi(z)$ from the compact set $\Phi(K)$ by an affine hyperplane in $\mathbb{R}^D$, which holds if and only if $\Phi(z) \notin \operatorname{conv}(\Phi(K))$ by the classical Hahn--Banach separation theorem.
\end{proof}

\subsection{A counterexample (failure of biconvex separation)}
The following example demonstrates that a biconvex set need not be biaffinely convex.

\begin{example}\label{ex:no}
	There exists a compact, connected biconvex set $K \subseteq \mathbb{R}^1 \times \mathbb{R}^1$ 
	such that $K \subsetneq \operatorname{biconv}(K)$. In particular, $K$ cannot be represented 
	as the intersection of any family (finite or infinite) of biaffine polynomial inequalities.
	
	Indeed, consider the  L-shaped domain in $\mathbb{R}^2 = \mathbb{R}^1 \times \mathbb{R}^1$,
	\[
	K = ([0,1] \times [0,2]) \cup ([0,2] \times [0,1]).
	\]
	The set $K$ is clearly a compact biconvex set.
	\begin{figure}[htbp]
		\centering
		\begin{tikzpicture}[scale=2.2, >=stealth]
			\draw[->] (-0.3, 0) -- (2.2, 0) node[right] {$x$};
			\draw[->] (0, -0.3) -- (0, 2.2) node[above] {$y$};
			
			\foreach \x in {1, 2} {
				\draw (\x, 0.04) -- (\x, -0.04) node[below] {\small $\x$};
			}
			\foreach \y in {1, 2} {
				\draw (0.04, \y) -- (-0.04, \y) node[left] {\small $\y$};
			}
			
			\fill[blue!12] (0,0) -- (2,0) -- (2,1) -- (1,1) -- (1,2) -- (0,2) -- cycle;
			\draw[thick, blue!60!black] (0,0) -- (2,0) -- (2,1) -- (1,1) -- (1,2) -- (0,2) -- cycle;
			
			
			\node[blue!60!black] at (0.5, 0.5) {\large $K$};
			
			\draw[red!80!black] (1.333, 0.04) -- (1.333, -0.04) node[below] {\footnotesize \color{red!80!black}$\frac43$};
			\draw[red!80!black] (0.04, 1.333) -- (-0.04, 1.333) node[left] {\footnotesize \color{red!80!black}$\frac43$};
			
			
			\node[circle, fill=red, draw=red!80!black, line width=0.8pt, inner sep=1pt, 
			label={[red!80!black]above right:{$z$}}] at (1.333, 1.333) {};
		\end{tikzpicture}
	\end{figure}
	Let $z = (x_z, y_z) = \left(\frac{4}{3}, \frac{4}{3}\right) \notin K$. Consider the canonical 
	biaffine embedding $\Phi: \mathbb{R}^2 \to \mathbb{R}^4$ given by $\Phi(x,y) = (1, x, y, xy)^T$. 
	Evaluating $\Phi$ at $z$ gives
	\[
	\Phi(z) = \Phi\left(\frac{4}{3}, \frac{4}{3}\right) = {
		\renewcommand{\arraystretch}{1.3}
		\begin{pmatrix} 
			1 \\ 
			\frac{4}{3} \\ 
			\frac{4}{3} \\ 
			\frac{16}{9} 
	\end{pmatrix}}.
	\]
	Now consider the following three points belonging to $K$:
	\[
	Q_0 = (0,0) \in K \implies \Phi(Q_0) = (1, 0, 0, 0)^T,
	\]
	\[
	Q_1 = (2,1) \in K \implies \Phi(Q_1) = (1, 2, 1, 2)^T,
	\]
	\[
	Q_2 = (1,2) \in K \implies \Phi(Q_2) = (1, 1, 2, 2)^T.
	\]
	Define the weights $\lambda_0 = \frac{1}{9}$ and $\lambda_1 = \lambda_2 = \frac{4}{9}$. 
	Clearly $\lambda_i > 0$ and $\lambda_0 + \lambda_1 + \lambda_2 = 1$. Moreover,
	\[
	{
		\renewcommand{\arraystretch}{1.3}
		\frac{1}{9}\Phi(Q_0) + \frac{4}{9}\Phi(Q_1) + \frac{4}{9}\Phi(Q_2) = \Phi(z),
	}
	\]
	which proves that $\Phi(z) \in \operatorname{conv}(\Phi(K))$. By Lemma~\ref{le:phieq}(a), 
	\[
	z \in \Phi^{-1}(\operatorname{conv}(\Phi(K))) = \operatorname{biconv}(K).
	\]
	Since $z = (4/3, 4/3) \notin K$, this establishes $K \subsetneq \operatorname{biconv}(K)$. By 
\end{example}

Example \ref{ex:no} demonstrates that the polynomials in $\mathcal{P}_{\text{biaff}}(K)$ cannot distinguish $K$ from its (potentially strictly larger) biaffine hull $\operatorname{biconv}(K)$. To obtain a category of sets whose geometry is characterized by biaffine polynomials, one must therefore restrict attention to sets that coincide with their biaffine hull. On this class, strict separation holds by construction.

\begin{proposition}
	Let $K \subseteq \mathbb{R}^{n+m}$ be a compact biaffinely convex set and let $z \in \mathbb{R}^{n+m} \setminus K$. Then there exists a biaffine polynomial $p \in \mathcal{P}_{\mathrm{biaff}}(\mathbb{R}^{n+m})$ such that
	\[
	p(w) \ge 0 \quad \text{for all } w \in K, \qquad \text{and} \quad p(z) < 0 .
	\]
\end{proposition}

\begin{proof}
	Since $K$ is biaffinely convex and $z \notin K$, we have $\Phi(z) \notin \operatorname{conv}(\Phi(K))$ by Definition~\ref{def:biaff_convex}. The existence of the separating biaffine polynomial then follows directly from Corollary~\ref{cor:sep}.
\end{proof}

\section{Kadison duality for biaffinely convex sets}\label{sec:3}
In the Kadison duality for partially convex sets, the dual function space $\CAff(K)$ forms a free  module of rank $n+1$ over the commutative $C^*$-algebra $\mathcal{C}(\pi_2(K))$. When passing to biaffine convexity, the geometric slices in both coordinates are restricted to affine dependencies. Consequently, the dual object changes from an infinite-dimensional module bundle to a finite-dimensional tensor-factorized Archimedean order unit (AOU) space of dimension $D = (n+1)(m+1)$. 
Throughout, for subsets $S \subseteq V$ and $T \subseteq W$ of real vector spaces $V$ and $W$, we denote the set of elementary (rank-1) tensors by
$$S \odot T := \{ v \otimes w \in V \otimes W \mid v \in S, \, w \in T \}.$$

\subsection{Abstract structure: biaffine order unit spaces}

To model functions that depend affinely on two separate coordinates, we view the dual space as a tensor product $E_1 \otimes E_2$. We require the order unit to factor as $u_1 \otimes u_2$ and the positive cone to be completely governed by coordinate point evaluations.  



\begin{definition}\label{def:bimodule}
	Let $(E_1, u_1, E_1^+)$ and $(E_2, u_2, E_2^+)$ be AOU spaces of dimensions $n+1$ and $m+1$, respectively. A biaffine order unit space of type $(n,m)$ is a tuple $(A, E_1, E_2, u, A^+)$ where $(A, u, A^+)$ is an AOU space 
	satisfying: 
	\begin{enumerate}
		\item 
		$A = E_1 \otimes E_2$ as a vector space, with distinguished order unit $u = u_1 \otimes u_2$.
		
		
		\item 
		The positive dual cone $A^{*+} = \{f \in A^* \mid f(A^+) \ge 0\}$ is generated by its rank-1  tensors,
		\[
		A^{*+} = \operatorname{cone}\Bigl( A^{*+} \cap \bigl(E_1^{*+} \odot E_2^{*+}\bigr) \Bigr).
		\]
		Equivalently, $a \in A^+$ if and only if $(\phi \otimes \psi)(a) \ge 0$ for every $\phi \otimes \psi \in A^{*+} \cap (E_1^{*+} \odot E_2^{*+})$.
		\item The marginal cones coincide with the slices of $A^+$ along the factor order units,
		\[
		E_1^+ = \{e \in E_1 \mid e \otimes u_2 \in A^+\} \quad \text{and} \quad E_2^+ = \{d \in E_2 \mid u_1 \otimes d \in A^+\}.
		\]
	\end{enumerate}
\end{definition}

\begin{remark} \label{re:ax2} The factorization $A = E_1 \otimes E_2$ is part of the defining data of $A$, reflecting the fixed two-party coordinate decomposition of the dual set.
	
	Moreover, note that Axiom (2) implies 
	\[
	E_1^+ \otimes_\epsilon E_2^+ \subseteq A^+,
	\]
	where $E_1^+ \otimes_\epsilon E_2^+ = \{a \in E_1 \otimes E_2 \mid (\phi \otimes \psi)(a) \ge 0 \ \forall \phi \in E_1^{*+}, \psi \in E_2^{*+}\}$ is the injective tensor cone.
\end{remark}

Before verifying that the space $\mathcal{P}_{\text{biaff}}(K)$ satisfies these abstract axioms, we prove two preliminary lemmas.

\begin{lemma}\label{le:cone}
	Let $K \subseteq \mathbb{R}^{n+m}$ be a compact set, and let $A = \mathcal{P}_{\mathrm{biaff}}(K)$ be equipped with the pointwise non-negative cone
	\[
	A^+ = \bigl\{ p \in A \bigm| p(x,y) \ge 0 \quad \forall (x,y) \in K \bigr\}
	\]
	and order unit $u = 1_K$. Let $\mathcal{E}(K) := \{ \operatorname{ev}_{(x,y)} \mid (x,y) \in K \} \subseteq A^*$ denote the set of evaluation functionals on $K$. 
	Then $A^{*+}$ equals the cone generated by $\mathcal{E}(K)$,
	\[
	A^{*+} := \{ \omega \in A^* \mid \omega(A^+) \ge 0 \} = \operatorname{cone}\bigl(\mathcal{E}(K)\bigr).
	\]
\end{lemma}

\begin{proof}
	Consider the cone generated by the evaluations, $C := \operatorname{cone}(\mathcal{E}(K)) \subseteq A^*$.
	A polynomial $p \in A$ belongs to $A^+$ if and only if $\operatorname{ev}_{(x,y)}(p) \ge 0$ for all $(x,y) \in K$, which holds if and only if $\omega(p) \ge 0$ for all $\omega \in \operatorname{cone}(\mathcal{E}(K))$. Therefore, identifying $A$ with $A^{**}$ via finite-dimensional reflexivity yields
	\[
	A^+ = \bigl\{ p \in A \bigm| \omega(p) \ge 0 \quad \forall \omega \in C \bigr\} = C^\circ,
	\]
	where $C^\circ$ denotes the polar dual of $C.$ By the bipolar theorem in finite dimensions, 
	\[
	A^{*+}= (A^+)^\circ = (C^\circ)^\circ = \overline{C} = \overline{\operatorname{cone}}\bigl(\mathcal{E}(K)\bigr),
	\]
	where $\overline{\operatorname{cone}}(\cdot)$ denotes the topological closure of the conic hull.
	
	Because $\omega(u) = 1$ for all $\omega \in \mathcal{E}(K)$, the set $\operatorname{conv}(\mathcal{E}(K))$ is a compact base of $C$ strictly separated from the origin ($0 \notin \operatorname{conv}(\mathcal{E}(K))$). In finite dimensions, the conical hull of any compact set not containing the origin is closed. Hence, $\overline{\operatorname{cone}}(\mathcal{E}(K)) = \operatorname{cone}(\mathcal{E}(K))$, yielding
	$
	A^{*+} = \operatorname{cone}\bigl(\mathcal{E}(K)\bigr).
	$
\end{proof}

\begin{proposition}\label{prop:fnsp}
	Let $K \subseteq \mathbb{R}^{n+m}$ be a compact non-degenerate set. Let $E_1 = \operatorname{Aff}(\pi_1(K))$ and $E_2 = \operatorname{Aff}(\pi_2(K))$ be the AOU spaces of continuous affine functions on the projections $\pi_1(K) \subset \mathbb{R}^n$ and $\pi_2(K) \subset \mathbb{R}^m$, equipped with units $u_1 = 1_{\pi_1(K)}$ and $u_2 = 1_{\pi_2(K)}$.
	
	Then the space of biaffine polynomials $\mathcal{P}_{\mathrm{biaff}}(K)$, equipped with the distinguished order unit $u = 1_K$ and the pointwise non-negative cone
	\[
	\mathcal{P}_{\mathrm{biaff}}(K)^+ = \bigl\{ p \in \mathcal{P}_{\mathrm{biaff}}(K) \bigm| p(x,y) \ge 0 \quad \forall (x,y) \in K \bigr\},
	\]
	is a biaffine order unit space of type $(n,m)$ over $(E_1, E_2)$.
\end{proposition}


\begin{proof} 
	Let $A = \mathcal{P}_{\mathrm{biaff}}(K)$. Because $K \subseteq \mathbb{R}^{n+m}$ is non-degenerate of type $(n,m)$, its marginal projections $\pi_1(K) \subseteq \mathbb{R}^n$ and $\pi_2(K) \subseteq \mathbb{R}^m$ affinely span $\mathbb{R}^n$ and $\mathbb{R}^m$, respectively. Thus, the standard affine coordinate functions $\{1, x_1, \dots, x_n\}$ and $\{1, y_1, \dots, y_m\}$ are linearly independent on $\pi_1(K)$ and $\pi_2(K)$, yielding $\dim(E_1) = n+1$ and $\dim(E_2) = m+1$. Because $K$ is compact, its projections $\pi_1(K)$ and $\pi_2(K)$ are compact sets; equipped with the supremum norms over their respective domains, $E_1$ and $E_2$ are finite-dimensional AOU spaces with order units $u_1 = 1_{\pi_1(K)}$ and $u_2 = 1_{\pi_2(K)}$.
	
	\textit{Axiom (1):} Consider the linear evaluation map induced by the tensor product
	\[
	\Theta \colon E_1 \otimes E_2 \longrightarrow \mathcal{P}_{\mathrm{biaff}}(K), \qquad \Theta(f \otimes g)(x,y) := f(x)g(y).
	\]
	By definition of biaffine polynomials, every $p \in \mathcal{P}_{\mathrm{biaff}}(K)$ expands as 
	\[
	p(x,y) = \sum_{i=0}^n \sum_{j=0}^m c_{ij} x_i y_j = \langle c, \Phi(x,y)\rangle \qquad \text{for all } (x,y) \in K
	\]
	with $x_0=y_0=1$, which proves that $\Theta$ is surjective. Furthermore, because $K$ is non-degenerate, the lifted set spans the full space, $\operatorname{span}(\Phi(K)) = \mathbb{R}^{(n+1)(m+1)}$. Thus, if $\langle c, \Phi(w)\rangle = 0$ for all $w \in K$, then $c = 0$, establishing that $\Theta$ is injective. 
	
	Hence, $\Theta$ is a linear isomorphism $E_1 \otimes E_2 \cong \mathcal{P}_{\mathrm{biaff}}(K)$, which identifies the distinguished order unit as
	\[
	u = (u_1 \otimes u_2)|_K = (1_{\pi_1(K)} \otimes 1_{\pi_2(K)})|_K = 1_K.
	\]
	Because $K$ is compact, $A = \mathcal{P}_{\mathrm{biaff}}(K)$ equipped with the pointwise non-negative cone $A^+ = \{p \in A \mid p(w) \ge 0 \text{ for all } w \in K\}$ and the supremum norm $\|p\|_\infty = \sup_{w \in K}|p(w)|$ is a real Archimedean order unit space of dimension $D = (n+1)(m+1)$.


	\textit{Axiom (2):} Note that for every $(x,y) \in K$, the point-evaluation functional $\operatorname{ev}_{(x,y)} \colon p \mapsto p(x,y)$ on $A$ factors as  rank-1 evaluation
	\[
	\operatorname{ev}_{(x,y)} = \operatorname{ev}_x \otimes \operatorname{ev}_y \in S(E_1) \odot S(E_2) \subseteq E_1^{*+} \odot E_2^{*+}.
	\]
	Set $\mathcal{E}(K) := \{\operatorname{ev}_x \otimes \operatorname{ev}_y \mid (x,y) \in K\}$. By definition of the cone $A^+$, an element $p \in A$ belongs to $A^+$ if and only if $\omega(p) \ge 0$ for all $\omega \in \mathcal{E}(K)$.

	Since $\mathcal{E}(K) \subseteq A^{*+} \cap (E_1^{*+} \odot E_2^{*+}) \subseteq A^{*+}$, we have
	\[
	A^{*+} = \operatorname{cone}\bigl(\mathcal{E}(K)\bigr) \subseteq \operatorname{cone}\Bigl( A^{*+} \cap \bigl(E_1^{*+} \odot E_2^{*+}\bigr) \Bigr) \subseteq A^{*+},
	\]
	where the first equality follows from Lemma \ref{le:cone}.
	Hence, $A^{*+} = \operatorname{cone}\bigl( A^{*+} \cap (E_1^{*+} \odot E_2^{*+}) \bigr)$, which confirms that $A^{*+}$ is generated by its rank-1  tensors.
	
	\textit{Axiom (3):} Let $e \in E_1 = \operatorname{Aff}(\pi_1(K))$. If $e \in E_1^+$, then $e(x) \ge 0$ for all $x \in \pi_1(K)$. Since $u_2 = 1_{\pi_2(K)}$, for every $(x,y) \in K \subseteq \pi_1(K) \times \pi_2(K)$ we have
	\[
	(e \otimes u_2)(x,y) = e(x)u_2(y) = e(x) \ge 0.
	\]
	Hence $e \otimes u_2 \in \mathcal{P}_{\mathrm{biaff}}(K)^+ = A^+$, proving $E_1^+ \subseteq \{e \in E_1 \mid e \otimes u_2 \in A^+\}$. 
	
	Conversely, suppose $e \in E_1$ satisfies $e \otimes u_2 \in A^+$. Then $(e \otimes u_2)(x,y) = e(x) \ge 0$ for all $(x,y) \in K$. By definition of the coordinate projection $\pi_1(K)$, for every $x \in \pi_1(K)$ there exists some $y \in \mathbb{R}^m$ such that $(x,y) \in K$. Evaluating at this pair yields $e(x) \ge 0$. Because this holds for every $x \in \pi_1(K)$, we have $e \in E_1^+$, establishing
	\[
	E_1^+ = \{e \in E_1 \mid e \otimes u_2 \in A^+\}.
	\]
	An identical argument with $u_1 = 1_{\pi_1(K)}$ and $d \in E_2 = \operatorname{Aff}(\pi_2(K))$ shows that $E_2^+ = \{d \in E_2 \mid u_1 \otimes d \in A^+\}$, completing the proof.
\end{proof}

\begin{remark}
	The proof of Proposition \ref{prop:fnsp} requires only the compactness and non-degeneracy of $K$, without using biaffine convexity. This parallels classical Kadison duality, where the space of affine functions on any compact set forms an Archimedean order unit space. Biaffine convexity of $K$ is required only in the forthcoming duality theorem (Theorem \ref{thm:dual}) to ensure the round-trip identification $K \cong K_{\mathcal{P}_{\mathrm{biaff}}(K)}$; for a general compact set, the state space reconstruction yields its biaffine hull $\mathrm{biconv}(K)$ rather than $K$ itself.
\end{remark}

\begin{definition}
	\label{def:biaff_iso}
	Let $(A, u_A, A^+)$ and $(B, u_B, B^+)$ be two biaffine order unit spaces of type $(n,m)$ over the factor AOU pairs $(E_1, E_2)$ and $(F_1, F_2)$, respectively, with order units $u_A = u_{E_1} \otimes u_{E_2}$ and $u_B = u_{F_1} \otimes u_{F_2}$.
	A linear map $\Phi \colon A \to B$ is an \textbf{isomorphism of biaffine order unit spaces} if
	\begin{enumerate}
		\item  $\Phi$ is a unital order isomorphism between the global AOU spaces, i.e., $\Phi$ is a linear bijection satisfying
		\[
		\Phi(u_A) = u_B \qquad \text{and} \qquad \Phi(A^+) = B^+.
		\]
		
		\item $\Phi$ respects the factor structure, meaning there exist unital order isomorphisms $\Phi_1 \colon E_1 \to F_1$ and $\Phi_2 \colon E_2 \to F_2$ such that
		$
		\Phi_1(u_{E_1}) = u_{F_1}, \Phi_2(u_{E_2}) = u_{F_2},
		$
		and
		\[
		\Phi = \Phi_1 \otimes \Phi_2.
		\]
	\end{enumerate}
\end{definition}

\subsection{The coordinate biaffine state space}

In classical Kadison duality, a compact convex set is recovered as the state space $S(A)$ of its function system. For an abstract biaffine order unit space $A \cong E_1 \otimes E_2$, however, the full state space $S(A)$ is convex, corresponding to the lifted convex hull $\operatorname{conv}(\Phi(K))$ rather than the set $K \subseteq \mathbb{R}^n \times \mathbb{R}^m$ itself. To reconstruct the original geometric set, one must isolate the unentangled product states $\phi \otimes \psi$, which represent genuine point evaluations $(x,y)$. By choosing affine coordinate bases for $E_1$ and $E_2$, these product states map back into $\mathbb{R}^{n+m}$, recovering $K$. 

\begin{definition} \label{def:statesp}
	Let $(E_1, u_1, E_1^+)$ and $(E_2, u_2, E_2^+)$ be finite-dimensional AOU spaces with state spaces $S(E_1) \subseteq E_1^*$ and $S(E_2) \subseteq E_2^*$. Let $(A, u, A^+)$ be an AOU space factoring linearly as $A \cong E_1 \otimes E_2$ with $u = u_1 \otimes u_2$.
	\begin{enumerate}
		
		\item The \textbf{biaffine state space} of $A$, denoted $S_{\mathrm{biaff}}(A)$, is the subset of product states
		\[
		S_{\mathrm{biaff}}(A) := S(A) \cap \bigl(S(E_1) \odot S(E_2)\bigr) = \bigl\{ \phi \otimes \psi \in S(E_1) \odot S(E_2) \bigm| (\phi \otimes \psi)(a) \ge 0 \;\; \forall a \in A^+ \bigr\}.
		\]
		
		\item 
		Given designated bases $\mathcal{B} = (\{e_i\}_{i=0}^n, \{d_j\}_{j=0}^m)$ of $E_1$ and $E_2$ (with $e_0 = u_1, d_0 = u_2$), the \textbf{coordinate biaffine state space} $K_{A, \mathcal{B}} \subseteq \mathbb{R}^{n+m}$ (often abbreviated as $K_A$ when the choice of basis is fixed) is the image of $S_{\text{biaff}}(A)$ under the 
		coordinate embedding $\iota \colon S(E_1) \odot S(E_2) \to \mathbb{R}^{n+m}$ given by
		\[
		\iota(\phi \otimes \psi) := \bigl(\phi(e_1), \dots, \phi(e_n), \psi(d_1), \dots, \psi(d_m)\bigr).
		\]
	\end{enumerate}
\end{definition}


\begin{lemma}
	\label{lem:intrinsic_dual_equiv}
	Let $(E_1, u_1, E_1^+)$ and $(E_2, u_2, E_2^+)$ be AOU spaces of dimensions $n+1$ and $m+1$, and let $(A, u, A^+)$ be an AOU space with $A \cong E_1 \otimes E_2$ and $u = u_1 \otimes u_2$. Then
	\[
	E_1^+ \otimes_\epsilon E_2^+ \subseteq A^+ \iff S(A) \subseteq \operatorname{conv}\bigl(S(E_1) \odot S(E_2)\bigr).
	\]
\end{lemma}

\begin{proof} Since $E_1^+ \otimes_\epsilon E_2^+$ and $A^+$ are closed (convex) cones, the bipolar theorem yields
	$$E_1^+ \otimes_\epsilon E_2^+ \subseteq A^+ \iff (A^+)^\circ \subseteq (E_1^+ \otimes_\epsilon E_2^+)^\circ.$$
	
	We claim that
	\[
	E_1^+ \otimes_\epsilon E_2^+ = \bigl\{ a \in E_1 \otimes E_2 \bigm| \omega(a) \ge 0 \quad \forall \omega \in S(E_1) \odot S(E_2) \bigr\}.
	\]
	Let $a \in E_1 \otimes E_2$. 
	If $(\phi \otimes \psi)(a) \ge 0$ for all $\phi \in E_1^{*+}$ and $\psi \in E_2^{*+}$, then in particular $(\phi \otimes \psi)(a) \ge 0$ for all normalized states $\phi \in S(E_1) \subset E_1^{*+}$ and $\psi \in S(E_2) \subset E_2^{*+}$.
	Conversely, suppose $(\tilde{\phi} \otimes \tilde{\psi})(a) \ge 0$ for all $\tilde{\phi} \in S(E_1)$ and $\tilde{\psi} \in S(E_2)$. Let $\phi \in E_1^{*+}$ and $\psi \in E_2^{*+}$ be arbitrary. If either $\phi = 0$ or $\psi = 0$, then $(\phi \otimes \psi)(a) = 0 \ge 0$. If both are non-zero, then $\phi(u_1) > 0$ and $\psi(u_2) > 0$. Factoring out the positive normalizations,
	\[
	\phi = \lambda \tilde{\phi} \quad \text{with } \lambda = \phi(u_1) > 0, \; \tilde{\phi} = \frac{\phi}{\phi(u_1)} \in S(E_1),
	\]
	\[
	\psi = \mu \tilde{\psi} \quad \text{with } \mu = \psi(u_2) > 0, \; \tilde{\psi} = \frac{\psi}{\psi(u_2)} \in S(E_2).
	\]
	By bilinearity of the tensor product,
	\[
	(\phi \otimes \psi)(a) = \lambda \mu \cdot (\tilde{\phi} \otimes \tilde{\psi})(a) \ge 0,
	\]
	since $\lambda \mu > 0$ and $(\tilde{\phi} \otimes \tilde{\psi})(a) \ge 0$. This proves the claim.

	Next, let $S = S(E_1) \odot S(E_2) \subseteq (E_1 \otimes E_2)^*$. It is easy to see that 
	\[
	\bigl(\operatorname{cone}(S)\bigr)^\circ := \bigl\{ a \in E_1 \otimes E_2 \bigm| \omega(a) \ge 0 \quad \forall \omega \in \operatorname{cone}(S) \bigr\} = \{a \in E_1 \otimes E_2 \mid \omega(a) \ge 0 \ \forall \omega \in S\}.
	\]
	Hence,
	\[
	E_1^+ \otimes_\epsilon E_2^+ = \bigl\{ a \in E_1 \otimes E_2 \bigm| \omega(a) \ge 0 \quad \forall \omega \in S(E_1) \odot S(E_2) \bigr\} = \bigl(\operatorname{cone}(S(E_1) \odot S(E_2))\bigr)^\circ.
	\]

	Since $S(E_1)$ and $S(E_2)$ are compact convex sets in finite dimensions, the set $S(E_1) \odot S(E_2)$ is compact and contained in the affine hyperplane $H = \{\omega \in A^* \mid \omega(u_1 \otimes u_2) = 1\}$. Because $0 \notin H$, the cone $C = \operatorname{cone}(S(E_1) \odot S(E_2))$ is closed. Applying the bipolar theorem in finite dimensions gives
	\[
	(E_1^+ \otimes_\epsilon E_2^+)^\circ = (C^\circ)^\circ = C = \operatorname{cone}\bigl(S(E_1) \odot S(E_2)\bigr).
	\]
	Another use of the bipolar theorem yields
	\[
	E_1^+ \otimes_\epsilon E_2^+ \subseteq A^+ \iff A^{*+} = A^{+\circ} \subseteq (E_1^+ \otimes_\epsilon E_2^+)^\circ = \operatorname{cone}\bigl(S(E_1) \odot S(E_2)\bigr) .
	\]
	Intersecting both sides with the affine hyperplane $H$ selects the state space bases
	\[
	A^{*+} \cap H = S(A) \qquad \text{and} \qquad \operatorname{cone}\bigl(S(E_1) \odot S(E_2)\bigr) \cap H = \operatorname{conv}\bigl(S(E_1) \odot S(E_2)\bigr).
	\]
	Therefore, $E_1^+ \otimes_\epsilon E_2^+ \subseteq A^+ \iff S(A) \subseteq \operatorname{conv}(S(E_1) \odot S(E_2))$.
\end{proof}

\begin{lemma}\label{le:extp}
	Let $(E_1, u_1, E_1^+)$ and $(E_2, u_2, E_2^+)$ be AOU spaces of dimensions $n+1$ and $m+1$, and let $(A, u, A^+)$ be an AOU space with $A \cong E_1 \otimes E_2$ and $u = u_1 \otimes u_2$. Then the following are equivalent:
	\begin{enumerate}[(i)]
		\item The positive dual cone is generated by its rank-$1$  tensors,
		\[
		A^{*+} = \operatorname{cone}\Bigl( A^{*+} \cap \bigl(E_1^{*+} \odot E_2^{*+}\bigr) \Bigr).
		\]
		\item Every extreme state of $A$ is a rank-$1$ product state,
		\[
		\operatorname{ext}(S(A)) \subseteq S(E_1) \odot S(E_2) \quad \bigl(\text{hence } \operatorname{ext}(S(A)) \subseteq S_{\mathrm{biaff}}(A)\bigr).
		\]
		\item The full state space is the convex hull of the biaffine state space,
		\[
		S(A) = \operatorname{conv}\bigl(S_{\mathrm{biaff}}(A)\bigr).
		\]
		\item Positivity in $A$ is completely determined on rank-$1$ states,
		\[
		a \in A^+ \iff (\phi \otimes \psi)(a) \ge 0 \quad \forall \phi \otimes \psi \in S_{\mathrm{biaff}}(A).
		\]
	\end{enumerate}
\end{lemma}

\begin{proof}
	\textit{(i) $\Rightarrow$ (ii):} Since $S(A) = A^{*+} \cap \{\omega \in A^* \mid \omega(u) = 1\}$, the state space $S(A)$ is a compact convex base of the pointed cone $A^{*+}$. Therefore, every extreme point of $S(A)$ generates an extreme ray of $A^{*+}$. Under (i), $A^{*+} = \operatorname{cone}\bigl(A^{*+} \cap (E_1^{*+} \odot E_2^{*+})\bigr)$, which implies that all extreme rays of $A^{*+}$ reside in $E_1^{*+} \odot E_2^{*+}$. Any non-zero functional $\omega \in E_1^{*+} \odot E_2^{*+}$ factors as $\omega = \phi \otimes \psi$ with non-zero $\phi \in E_1^{*+}$ and $\psi \in E_2^{*+}$, so $\phi(u_1) > 0$ and $\psi(u_2) > 0$. 
	Normalizing we get
	\[
	\tilde{\omega} = \left(\frac{\phi}{\phi(u_1)}\right) \otimes \left(\frac{\psi}{\psi(u_2)}\right) \in S(E_1) \odot S(E_2).
	\]
	Thus $\operatorname{ext}(S(A)) \subseteq S(E_1) \odot S(E_2)$. Because $\operatorname{ext}(S(A)) \subseteq S(A)$, we have
	\[
	\operatorname{ext}(S(A)) \subseteq S(A) \cap \bigl(S(E_1) \odot S(E_2)\bigr) = S_{\mathrm{biaff}}(A).
	\]
	
	\textit{(ii) $\Rightarrow$ (iii):} Since $S(A)$ is a compact convex set, Minkowski's theorem and \textit{(ii)} give
	\[
	S(A) = \operatorname{conv}(\operatorname{ext}(S(A))) \subseteq \operatorname{conv}(S_{\mathrm{biaff}}(A)) \subseteq S(A),
	\]
	which forces $S(A) = \operatorname{conv}(S_{\mathrm{biaff}}(A))$.
	
	\textit{(iii) $\Rightarrow$ (i):} Because $S(A)$ is a base of $A^{*+}$, we have $A^{*+} = \operatorname{cone}(S(A))$. Substituting $S(A) = \operatorname{conv}(S_{\mathrm{biaff}}(A))$ yields
	\[
	A^{*+} = \operatorname{cone}\bigl(\operatorname{conv}(S_{\mathrm{biaff}}(A))\bigr) = \operatorname{cone}(S_{\mathrm{biaff}}(A)).
	\]
	Since $S_{\mathrm{biaff}}(A) \subseteq A^{*+} \cap (E_1^{*+} \odot E_2^{*+}) \subseteq A^{*+}$, taking conic hulls yields
	\[
	A^{*+} = \operatorname{cone}\Bigl(A^{*+} \cap \bigl(E_1^{*+} \odot E_2^{*+}\bigr)\Bigr),
	\]
	establishing \textit{(i)} and closing the equivalence $(i) \Leftrightarrow (ii) \Leftrightarrow (iii)$.
	
	\textit{(iii) $\Rightarrow$ (iv):} In any Archimedean order unit space, positivity is detected by evaluation against states,
	\[
	a \in A^+ \iff \omega(a) \ge 0 \quad \forall \omega \in S(A).
	\]
	Under \textit{(iii)}, every state $\omega \in S(A)$ is a finite convex combination $\omega = \sum_{k=1}^N \lambda_k \omega_k$ of elements $\omega_k \in S_{\mathrm{biaff}}(A)$ with $\lambda_k \ge 0$ and $\sum_{k=1}^N \lambda_k = 1$. By linearity, $\omega(a) \ge 0$ for all $\omega \in S(A)$ if and only if $\omega(a) \ge 0$ for all $\omega \in S_{\mathrm{biaff}}(A)$.
	
	\textit{(iv) $\Rightarrow$ (iii):}  Since $S_{\mathrm{biaff}}(A) \subseteq S(A)$ and $S(A)$ is convex, the inclusion $\operatorname{conv}(S_{\mathrm{biaff}}(A)) \subseteq S(A)$ is immediate. For the reverse inclusion, note that $S(E_1)$ and $S(E_2)$ are compact in finite dimensions, making $S(E_1) \odot S(E_2)$ compact in $A^*$. Because $S(A)$ is closed, $S_{\mathrm{biaff}}(A) = S(A) \cap (S(E_1) \odot S(E_2))$ is compact, so $C := \operatorname{conv}(S_{\mathrm{biaff}}(A))$ is a compact convex subset of $S(A)$.
	
	Suppose for contradiction that there exists a state $\omega_0 \in S(A) \setminus C$. By the Hahn-Banach separation theorem in $A^*$, there exist $a \in A$ and $\alpha \in \mathbb{R}$ such that
	\begin{equation}\label{eq:sep1}
		\omega(a) \ge \alpha \quad \forall \omega \in C, \qquad \text{and} \qquad \omega_0(a) < \alpha.
	\end{equation}
	Define the shifted element $\tilde{a} := a - \alpha u \in A$. Since every functional $\omega \in C \subseteq S(A)$ satisfies $\omega(u) = 1$, we have
	\[
	\omega(\tilde{a}) = \omega(a) - \alpha \ge 0 \quad \forall \omega \in S_{\mathrm{biaff}}(A) \subseteq C.
	\]
	By \textit{(iv)}, this  implies $\tilde{a} \in A^+$. However, since $\omega_0 \in S(A) \subset A^{*+}$, it must evaluate non-negatively on all elements of $A^+.$ Now
	\[
	\omega_0(\tilde{a}) = \omega_0(a) - \alpha < 0
	\]
	from \eqref{eq:sep1} contradicts $\tilde{a} \in A^+$. Thus no such state $\omega_0$ can exist, forcing $S(A) \subseteq \operatorname{conv}(S_{\mathrm{biaff}}(A))$ and concluding that $S(A) = \operatorname{conv}(S_{\mathrm{biaff}}(A))$.
\end{proof}

\begin{remark}
	Lemma~\ref{le:extp} establishes the algebraic bridge ensuring that an abstract biaffine order unit space $A$ behaves like a space of functions on a geometric set of coordinate pairs $(x,y) \in \mathbb{R}^{n+m}$. In a general tensor product $E_1 \otimes E_2$, the dual space can contain higher-rank (or ``entangled'') extreme states that do not correspond to any evaluation at a geometric point. The equivalent conditions of Lemma~\ref{le:extp} rule out such non-geometric behavior.
	
	Condition \textit{(ii)} says that every extreme state of $S(A)$ is an unentangled rank-$1$ product state, guaranteeing that the extreme boundary consists entirely of coordinate points.
	By condition \textit{(iii)}, the full state space contains no non-separable states protruding beyond the convex hull of classical product states, so $S(A) = \operatorname{conv}(S_{\mathrm{biaff}}(A))$.
	Finally, condition \textit{(iv)} reflects the concrete function-theoretic principle that positivity of an element $a \in A$ is detected purely pointwise on geometric pairs $(x,y)$, without needing to test against non-geometric states.
\end{remark}

\begin{remark}\label{rem:namioka_phelps}
	The inclusion $E_1^+ \otimes_{\epsilon} E_2^+ \subseteq A^+$ from Remark \ref{re:ax2} connects our setting to the classical theory of ordered tensor products by Namioka and Phelps \cite{NP69}. For ordered vector 
	spaces with order units, the minimal and maximal compatible tensor cones are the projective 
	cone $E_1^+ \otimes_\pi E_2^+ = \operatorname{cone}(E_1^+ \odot E_2^+)$ and the injective cone 
	\[
	E_{1}^{+}\otimes_{\epsilon}E_{2}^{+} = \left\{ a \in E_1 \otimes E_2 \;\middle|\; (\phi \otimes \psi)(a) \ge 0 \; \forall \phi \in E_1^{*+}, \, \psi \in E_2^{*+} \right\}.
	\]
	In \cite{NP69}, Namioka and Phelps proved that the state space of the injective tensor product 
	$E_1 \otimes_\epsilon E_2$ is precisely the convex hull of the product of the state spaces,
	\[
	S(E_1 \otimes_\epsilon E_2) = \operatorname{conv}(S(E_1) \odot S(E_2)).
	\]
	In quantum information and convex geometry, this is the set of all separable states.
	Geometrically, $E_1 \otimes_\epsilon E_2$ corresponds to the full Cartesian product  $\pi_1(K) \times \pi_2(K)$. By Remark \ref{re:ax2}, any biaffine order unit space $A$ satisfies
	\[
	E_1^+ \otimes_\epsilon E_2^+ \subseteq A^+,
	\]
	which by Lemma \ref{lem:intrinsic_dual_equiv} means that the state space $S(A) = \operatorname{conv}(S_{\mathrm{biaff}}(A))$ is a 
	closed convex subset of the separable state space $\operatorname{conv}(S(E_1) \odot S(E_2))$. When $K = \pi_1(K) \times \pi_2(K)$ is a product domain, 
	$A$ coincides with the injective tensor product $E_1 \otimes_\epsilon E_2$. 
	For a general non-product biaffinely convex set $K \subsetneq \pi_1(K) \times \pi_2(K)$, $A^+$ 
	is strictly larger than $E_1^+ \otimes_\epsilon E_2^+$, reflecting that polynomials need only 
	be non-negative on $K$ rather than the ambient product box. 
\end{remark}

\begin{lemma}\label{le:lambda}
	Let $(A, u, A^+)$ be a biaffine order unit space of type $(n,m)$ over $(E_1, E_2)$ with designated bases $\{e_i\}_{i=0}^n$ and $\{d_j\}_{j=0}^m$ satisfying $e_0 = u_1$ and $d_0 = u_2$. Let $H_1 = \{\phi \in E_1^* \mid \phi(u_1) = 1\}$ and $H_2 = \{\psi \in E_2^* \mid \psi(u_2) = 1\}$. 
	
	Define the linear coordinate map $\Lambda: A^* \to \mathbb{R}^{(n+1)(m+1)}$ by $\Lambda(\omega) := (\omega(e_i \otimes d_j))_{i=0,j=0}^{n,m}$. Then, with the coordinate map $\iota(\phi \otimes \psi) := ((\phi(e_i))_{i=1}^n, (\psi(d_j))_{j=1}^m)$ and the biaffine embedding $\Phi(x,y) := \binom{1}{x} \otimes \binom{1}{y},$
	\begin{enumerate}[(a)]
		\item $\Lambda$ is a linear isomorphism satisfying $\Lambda(\phi \otimes \psi) = \Phi(\iota(\phi \otimes \psi))$ on $H_1 \odot H_2$, and hence $\Lambda(H_1 \odot H_2) = \Sigma_{n,m}$,
		\item $S(A) \cap (H_1 \odot H_2) = S_{\mathrm{biaff}}(A)$,
		\item $\Lambda(S(A)) \cap \Sigma_{n,m} = \Phi(K_A)$.
	\end{enumerate}
\end{lemma}

\begin{proof}
	Because $e_0 = u_1$ and $d_0 = u_2$, any functional $\phi \in H_1$ satisfies $(\phi(e_0), \dots, \phi(e_n))^T = \binom{1}{x}$ for some $x \in \mathbb{R}^n$, and any $\psi \in H_2$ satisfies $(\psi(d_0), \dots, \psi(d_m))^T = \binom{1}{y}$ for some $y \in \mathbb{R}^m$. The image under $\Lambda$ of the set of  rank-$1$ functionals $H_1 \odot H_2 := \{\phi \otimes \psi \mid \phi \in H_1, \psi \in H_2\}$ is precisely the Segre variety
	\[
	\Lambda(H_1 \odot H_2) = \left\{ \binom{1}{x} \otimes \binom{1}{y} \;\middle|\; x \in \mathbb{R}^n, \, y \in \mathbb{R}^m \right\} = \Sigma_{n,m}.
	\]
	
	More generally, for any $\phi \in H_1$ and $\psi \in H_2$ with $x = \iota_1(\phi) \in \mathbb{R}^n$ and $y = \iota_2(\psi) \in \mathbb{R}^m$, we have
	\[
	\Lambda(\phi \otimes \psi) = \binom{1}{x} \otimes \binom{1}{y} = \Phi(x, y) = \Phi(\iota(\phi \otimes \psi)).
	\]
	Since $\{e_i \otimes d_j\}_{i=0, j=0}^{n, m}$ is a basis for $A \cong E_1 \otimes E_2$, the map $\Lambda$ is the coordinate isomorphism associated with its dual basis, which establishes \textit{(a)}.
	
	Next, we show \textit{(b)}. The inclusion $(\supseteq)$ is immediate because $S(E_1) \subset H_1$ and $S(E_2) \subset H_2$. For the reverse inclusion, let $\phi \otimes \psi \in S(A) \cap (H_1 \odot H_2)$. For any $e \in E_1^+$, Axiom (2) (cf.~Remark \ref{re:ax2}) ensures that $e \otimes u_2 \in E_1^+ \otimes_\epsilon E_2^+ \subseteq A^+$. Since $\phi \otimes \psi \in S(A) \subset A^{*+}$, we have
	\[
	(\phi \otimes \psi)(e \otimes u_2) = \phi(e)\psi(u_2) = \phi(e) \ge 0,
	\]
	using $\psi(u_2) = 1$. Thus $\phi \in E_1^{*+}$, which together with $\phi(u_1) = 1$ implies $\phi \in S(E_1)$. By an identical argument testing against $u_1 \otimes d$ for $d \in E_2^+$, we obtain $\psi \in S(E_2)$. Therefore, $\phi \otimes \psi \in S(A) \cap (S(E_1) \odot S(E_2)) = S_{\mathrm{biaff}}(A)$, proving \textit{(b)}.
	
	Finally, for \textit{(c)}, since $\Lambda$ is injective, it preserves intersections of subsets. Using \textit{(a)} and \textit{(b)}, we obtain
	\[
	\Lambda(S(A)) \cap \Sigma_{n,m} = \Lambda(S(A)) \cap \Lambda(H_1 \odot H_2) = \Lambda(S(A) \cap (H_1 \odot H_2)) = \Lambda(S_{\mathrm{biaff}}(A)) = \Phi(K_A),
	\]
	which completes the proof.
\end{proof}

\begin{proposition}\label{prop:sets}
	Let $(A, u, A^+)$ be a biaffine order unit space of type $(n,m)$ over $(E_1, E_2)$ with designated affine bases $\{e_i\}_{i=0}^n$ of $E_1$ (with $e_0 = u_1$) and $\{d_j\}_{j=0}^m$ of $E_2$ (with $d_0 = u_2$). 
	Then the coordinate biaffine state space
	\[
	K_A = \iota\bigl(S_{\mathrm{biaff}}(A)\bigr) \subseteq \mathbb{R}^{n+m}
	\]
	is a compact non-degenerate biaffinely convex set.
\end{proposition}

\begin{proof}
	\textit{Compactness}: Because $E_1$ and $E_2$ are finite-dimensional AOU spaces, 
	their state spaces $S(E_1) \subseteq E_1^*$ and $S(E_2) \subseteq E_2^*$ are compact 
	convex sets. Consequently, the set of rank-1 product states 
	\[
	S(E_1) \odot S(E_2) := \{\phi \otimes \psi \mid \phi \in S(E_1), \, \psi \in S(E_2)\} \subseteq A^*
	\]
	is compact. The global state space $S(A) = \{\omega \in A^{*+} \mid \omega(u) = 1\}$ 
	is a compact convex base of the pointed closed cone $A^{*+}$. Therefore, the biaffine 
	state space
	\[
	S_{\mathrm{biaff}}(A) = S(A) \cap (S(E_1) \odot S(E_2))
	\]
	is compact as the intersection of a closed set and a compact set. 
	
	To see that the coordinate evaluation map $\iota: S(E_1) \odot S(E_2) \to \mathbb{R}^{n+m}$ 
	is continuous, recall from Lemma~\ref{le:lambda}(a) that 
	$\Lambda(\phi \otimes \psi) = \Phi(\iota(\phi \otimes \psi))$. Since the canonical 
	embedding $\Phi: \mathbb{R}^{n+m} \to \Sigma_{n,m}$ has a continuous inverse given by 
	coordinate projection, on product states we have
	\[
	\iota = \Phi^{-1} \circ \Lambda.
	\]
	Because $\Lambda: A^* \to \mathbb{R}^{(n+1)(m+1)}$ is linear (hence continuous) and 
	$\Phi^{-1}$ is continuous, $\iota$ is continuous on $S(E_1) \odot S(E_2)$. 
	Consequently, its image $K_A = \iota(S_{\mathrm{biaff}}(A))$ is a compact subset of $\mathbb{R}^{n+m}$.
	
	\textit{Biaffine convexity:} 
	By Lemma \ref{le:extp} \textit{(iii)}, $S(A) = \operatorname{conv}(S_{\mathrm{biaff}}(A))$. Applying the linear coordinate isomorphism $\Lambda$ from Lemma \ref{le:lambda} yields
	\[
	\Lambda(S(A)) = \Lambda\bigl(\operatorname{conv}(S_{\mathrm{biaff}}(A))\bigr) = \operatorname{conv}\bigl(\Lambda(S_{\mathrm{biaff}}(A))\bigr) = \operatorname{conv}(\Phi(K_A)) \text{}.
	\]
	Intersecting both sides with the Segre variety $\Sigma_{n,m}$ and applying Lemma \ref{le:lambda} \textit{(c)} gives
	\[
	\operatorname{conv}(\Phi(K_A)) \cap \Sigma_{n,m} = \Lambda(S(A)) \cap \Sigma_{n,m} = \Phi(K_A) \text{}.
	\]
	By Lemma \ref{le:phieq}, this equality is equivalent to $K_A = \operatorname{biconv}(K_A)$, so $K_A$ is biaffinely convex.
	
	\textit{Non-degeneracy of $K_A$:} 
	Because $A$ is an Archimedean order unit space, its positive cone $A^+$ is closed and pointed ($A^+ \cap (-A^+) = \{0\}$). In finite dimensions, the dual cone of a closed, pointed cone is generating, which ensures
	$$\operatorname{span}(A^{*+}) = A^* = E_1^* \otimes E_2^*$$
	By Axiom (2) (or Lemma \ref{le:extp}),  $A^{*+} = \operatorname{cone}(S_{\text{biaff}}(A))$. Consequently,
	$$\operatorname{span}(S_{\text{biaff}}(A)) = \operatorname{span}(A^{*+}) = A^*$$
	Recall the linear coordinate isomorphism $\Lambda: A^* \to \mathbb{R}^{(n+1)(m+1)}$ given by $\Lambda(\omega) := (\omega(e_i \otimes d_j))_{i=0, j=0}^{n, m}$. On $S_{\text{biaff}}(A)$, we have $\Lambda(\phi \otimes \psi) = \Phi(\iota(\phi \otimes \psi))$, meaning $\Lambda(S_{\text{biaff}}(A)) = \Phi(K_A)$. Applying $\Lambda$ gives
	$$\operatorname{span}(\Phi(K_A)) = \operatorname{span}(\Lambda(S_{\text{biaff}}(A))) = \Lambda(\operatorname{span}(S_{\text{biaff}}(A))) = \Lambda(A^*) = \mathbb{R}^{(n+1)(m+1)}.$$
	Thus, $K_A$ is non-degenerate.
\end{proof}

\begin{lemma}\label{le:basis-change}
	Let $A$ be a biaffine order unit space of type $(n,m)$ over $(E_1, E_2)$. Let $\mathcal{B} = (\{e_i\}_{i=0}^n, \{d_j\}_{j=0}^m)$ and $\mathcal{B}' = (\{e'_i\}_{i=0}^n, \{d'_j\}_{j=0}^m)$ be two choices of designated affine bases for $E_1$ and $E_2$ (with $e_0 = e'_0 = u_1$ and $d_0 = d'_0 = u_2$), and let $K_A, K'_A \subseteq \mathbb{R}^{n+m}$ be the corresponding coordinate biaffine state spaces.  Then $K_A$ and $K'_A$ are isomorphic biaffinely convex sets.
\end{lemma}

\begin{proof}
	Because $\{e_i\}_{i=0}^n$ and $\{e'_i\}_{i=0}^n$ are both bases of $E_1$ with $e_0 = e'_0 = u_1$, each basis element $e'_k$ ($1 \le k \le n$) uniquely expands as
	$$e'_k = \alpha_{k0} u_1 + \sum_{i=1}^n \alpha_{ki} e_i$$
	for coefficients $\alpha_{ki} \in \mathbb{R}$. 
	For any $\phi_0 \in S(E_1)$, we have $\phi_0(u_1) = 1$ and thus
	$$\phi_0(e'_k) = \alpha_{k0}\phi_0(u_1) + \sum_{i=1}^n \alpha_{ki}\phi_0(e_i) = \alpha_{k0} + \sum_{i=1}^n \alpha_{ki}\phi_0(e_i).$$
	Setting $x = (\phi_0(e_1), \dots, \phi_0(e_n))^\top \in \mathbb{R}^n$ and $x' = (\phi_0(e'_1), \dots, \phi_0(e'_n))^\top \in \mathbb{R}^n$, this coordinate change takes the form
	$$x' = M x + v,$$
	where $M = (\alpha_{ki})_{k,i=1}^n \in \mathbb{R}^{n \times n}$ and $v = (\alpha_{k0})_{k=1}^n \in \mathbb{R}^n$. 
	The transition matrix between the full bases $\{e_i\}_{i=0}^n$ and $\{e'_i\}_{i=0}^n$ is the block matrix
	$$T = \begin{pmatrix} 1 & 0 \\ v & M \end{pmatrix}.$$
	Since $T$ is invertible, $\det(T) = \det(M) \neq 0$, which proves that $M$ is invertible. Hence, the map $\phi: \mathbb{R}^n \to \mathbb{R}^n$, $\phi(x) := Mx + v$, is an affine isomorphism.
	
	By an identical argument on $E_2$ with $d_0 = d'_0 = u_2$, there exists an affine isomorphism $\psi: \mathbb{R}^m \to \mathbb{R}^m$ such that $\psi((\psi_0(d_1), \dots, \psi_0(d_m))^\top) = (\psi_0(d'_1), \dots, \psi_0(d'_m))^\top$ for all $\psi_0 \in S(E_2)$. Defining the product map $\Theta(x,y) := (\phi(x), \psi(y))$, we have by construction that $\Theta(\iota_{\mathcal{B}}(\omega)) = \iota_{\mathcal{B}'}(\omega)$ for every rank-1 state $\omega \in S(E_1) \odot S(E_2)$. In particular, $\Theta(K_A) = K'_A$. By Definition \ref{def:biiso}, $K_A$ and $K'_A$ are isomorphic biaffinely convex sets. 
\end{proof}

\subsection{The duality theorem} We now unite the forward assignment $K \mapsto \mathcal{P}_{biaff}(K)$ and the state space reconstruction $A \mapsto K_A$ to establish our main result, Theorem \ref{thm:dual}, a Kadison duality for biaffinely convex sets. We split the proof into two parts:~the geometric isomorphism ($K \cong K_{\mathcal{P}_{biaff}(K)}$) in Proposition \ref{prop:setiso} and the algebraic isomorphism ($A \cong \mathcal{P}_{biaff}(K_A)$) in Proposition \ref{prop:fniso}.


\begin{proposition}\label{prop:setiso}
	Let $K \subseteq \R^{n+m}$ be a compact non-degenerate biaffinely convex set. Then $\PhiAff(K)$ is a biaffine order unit space of type $(n,m)$ and $K$ is isomorphic to $K_{\PhiAff(K)}.$
\end{proposition}

\begin{proof}
	Let $K \subseteq \mathbb{R}^{n+m}$ be a compact non-degenerate biaffinely convex set. By Proposition \ref{prop:fnsp}, $A := \mathcal{P}_{\mathrm{biaff}}(K)$ is a biaffine order unit space of type $(n,m)$ over $E_1 = \operatorname{Aff}(\pi_1(K))$ and $E_2 = \operatorname{Aff}(\pi_2(K))$. Equip $E_1$ and $E_2$ with their canonical affine coordinate bases
	\[
	e_0 = u_1, \quad e_i(x) = x_i \ (1 \le i \le n); \qquad d_0 = u_2, \quad d_j(y) = y_j \ (1 \le j \le m).
	\]
	By Lemma \ref{le:cone},  $S(A) = \operatorname{conv}(\mathcal{E}(K))$, where $\mathcal{E}(K) := \{ \operatorname{ev}_{(x,y)} \mid (x,y) \in K \}$. Under the coordinate map $\Lambda$ from Lemma \ref{le:lambda}, each evaluation functional satisfies $\Lambda(\operatorname{ev}_{(x,y)}) = \Phi(x,y)$, so $\Lambda(\mathcal{E}(K)) = \Phi(K)$. By linearity of $\Lambda$,
	\[
	\Lambda(S(A)) = \Lambda(\operatorname{conv}(\mathcal{E}(K))) = \operatorname{conv}(\Lambda(\mathcal{E}(K))) = \operatorname{conv}(\Phi(K)).
	\]
	Intersecting both sides with the Segre variety $\Sigma_{n,m}$ and applying Lemma \ref{le:lambda} \textit{(c)} yields
	\[
	\operatorname{conv}(\Phi(K)) \cap \Sigma_{n,m} = \Lambda(S(A)) \cap \Sigma_{n,m} = \Phi(K_A).
	\]
	Since $K$ is biaffinely convex, Lemma \ref{le:phieq} gives $\operatorname{conv}(\Phi(K)) \cap \Sigma_{n,m} = \Phi(K)$. Therefore, $\Phi(K) = \Phi(K_A)$. Because the embedding $\Phi: \mathbb{R}^{n+m} \to \Sigma_{n,m}$ is a bijection, this implies $K = K_A$. The identity map is therefore an isomorphism between $K$ and $K_{\mathcal{P}_{\mathrm{biaff}}(K)}$.
\end{proof}

\begin{remark}
	An inspection of the proof of Proposition \ref{prop:setiso} shows that if  $K \subseteq \R^{n+m}$ is only assumed compact and non-degenerate, then $K_{\PhiAff(K)} = \text{biconv}(K).$
\end{remark}

\begin{proposition}\label{prop:fniso}
	Let $A$ be an abstract biaffine order unit space of type $(n,m)$ over $(E_1,E_2)$. 
	For any choice of designated affine coordinate bases $\{e_i\}_{i=0}^n$ of $E_1$ 
	(with $e_0 = u_1$) and $\{d_j\}_{j=0}^m$ of $E_2$ (with $d_0 = u_2$), the associated 
	coordinate biaffine state space $K_A \subseteq \mathbb{R}^{n+m}$ is a compact 
	non-degenerate biaffinely convex set, and $A$ is isomorphic to $\mathcal{P}_{\mathrm{biaff}}(K_A)$.
\end{proposition}

\begin{proof}
	Let $(A, E_1,E_2, u, A^+)$ be a biaffine order unit space of type $(n,m)$ over $(E_1, E_2)$ 
	equipped with designated affine bases $\{e_i\}_{i=0}^n$ and $\{d_j\}_{j=0}^m$. 
	By Proposition~\ref{prop:sets}, the coordinate biaffine state space 
	$K_A \subseteq \mathbb{R}^{n+m}$ is a compact non-degenerate biaffinely convex set. 
	We show that $A \cong \mathcal{P}_{\mathrm{biaff}}(K_A)$ in the sense of Definition~\ref{def:biaff_iso}.
	
	Set $\iota_1(\varphi)=(\varphi(e_i))_i, \iota_2(\psi)=(\psi(d_j))_j$ and define the maps
	\begin{align*}
		\Psi_1: E_1 &\longrightarrow \operatorname{Aff}(\pi_1(K_A)), \quad (\Psi_1(e))(x) := \phi_x(e), \quad \text{where } \phi_x = \iota_1^{-1}(x) \in S(E_1), \\
		\Psi_2: E_2 &\longrightarrow \operatorname{Aff}(\pi_2(K_A)), \quad (\Psi_2(d))(y) := \psi_y(d), \quad \text{where } \psi_y = \iota_2^{-1}(y) \in S(E_2) .
	\end{align*}
	By construction, $\Psi_1(u_1) = 1_{\pi_1(K_A)}$, $\Psi_1(e_i)(x) = x_i$ for $i \ge 1$, $\Psi_2(u_2) = 1_{\pi_2(K_A)}$, and $\Psi_2(d_j)(y) = y_j$ for $j \ge 1$. Hence $\Psi_1$ and $\Psi_2$ are unital linear bijections.
	
	We claim that $\Psi_1$ and $\Psi_2$ are order isomorphisms between AOU spaces. By Axiom (3),
	\[
	e \in E_1^+ \iff e \otimes u_2 \in A^+ .
	\]
	The map $\iota$ restricts to a bijection between $S_{\mathrm{biaff}}(A)$ and $K_A$, so every element of $S_{\mathrm{biaff}}(A)$ is uniquely of the form $\phi_x \otimes \psi_y$ for $(x,y) \in K_A.$ By Lemma \ref{le:extp} \textit{(iv)}, an element belongs to $A^+$ if and only if it evaluates non-negatively on all elements of $S_{\mathrm{biaff}}(A)$. For  $e \otimes u_2,$ this means
	\begin{align*}
		e \otimes u_2 \in A^+ &\iff (\phi_x \otimes \psi_y)(e \otimes u_2) \ge 0 \quad \forall (x,y) \in K_A = \iota(S_{\mathrm{biaff}}(A)) \\
		&\iff \phi_x(e)\psi_y(u_2) \ge 0 \quad \forall (x,y) \in K_A \\
		&\iff (\Psi_1(e))(x) \ge 0 \quad \forall x \in \pi_1(K_A) \\
		&\iff \Psi_1(e) \in \operatorname{Aff}(\pi_1(K_A))^+.
	\end{align*}
	Thus $\Psi_1(E_1^+) = \operatorname{Aff}(\pi_1(K_A))^+$, proving that $\Psi_1$ is a unital order isomorphism. An identical calculation with $u_1 \otimes d$ shows that $\Psi_2(E_2^+) = \operatorname{Aff}(\pi_2(K_A))^+$.
	
	
	Finally, consider the map
	$$\mu : \operatorname{Aff}(\pi_1(K_A)) \otimes \operatorname{Aff}(\pi_2(K_A)) \longrightarrow \mathcal{P}_{\mathrm{biaff}}(K_A), \quad \mu(f \otimes g)(x,y) := f(x)g(y).$$
	Since $K_A$ is non-degenerate (Proposition \ref{prop:sets}), the argument from Axiom (1) of Proposition \ref{prop:fnsp} shows that $\mu$ is a linear isomorphism. We define the map $\Psi$ as the composition
	$$\Psi := \mu \circ (\Psi_1 \otimes \Psi_2) : A \cong E_1 \otimes E_2 \longrightarrow \mathcal{P}_{\mathrm{biaff}}(K_A).$$
	Under this map, for any $a \in A$ and $(x,y) \in K_A$, we have$$\Psi(a)(x,y) = (\phi_x \otimes \psi_y)(a).$$
	The map $\Psi$ is unital because $\Psi(u)(x,y) = (\phi_x \otimes \psi_y)(u_1 \otimes u_2) = 1$, so $\Psi(u) = 1_{K_A}$.  Moreover, $\Psi$ is both order-preserving and order-reflecting. 
	By Lemma \ref{le:extp} (iv),
	\begin{align*}
		a \in A^+ &\iff (\phi_x \otimes \psi_y)(a) \ge 0 \quad \forall (x,y) \in K_A\\ &\iff \Psi(a)(x,y) \ge 0 \quad \forall (x,y) \in K_A \\
		&\iff \Psi(a) \in \mathcal{P}_{\mathrm{biaff}}(K_A)^+.
	\end{align*}  
	Thus $\Psi(A^+) = \mathcal{P}_{\mathrm{biaff}}(K_A)^+$. Since $\Psi$ is a unital order isomorphism 
	whose factor maps $\Psi_1, \Psi_2$ are unital order isomorphisms, $\Psi$ is an isomorphism 
	of biaffine order unit spaces in the sense of Definition \ref{def:biaff_iso}.
\end{proof}

\section{Categorical duality}\label{sec:4}

Theorem~\ref{thm:dual} established an object-wise duality identifying any compact non-degenerate biaffinely convex set $K$ with the coordinate state space of its biaffine polynomials, and conversely, identifying any abstract biaffine order unit space $A$ with the biaffine polynomials on $K_A$. In this section, we elevate this correspondence to a contravariant categorical equivalence
\[
\mathbf{BiAffConv}^{\mathrm{op}} \simeq \mathbf{BiAOU}.
\]

\subsection{The Categories $\mathbf{BiAffConv}$ and $\mathbf{BiAOU}$}

We begin by defining the appropriate morphisms for both geometric and algebraic categories.

\begin{definition}
	The category of biaffinely convex sets, denoted by $\mathbf{BiAffConv}$, consists of:
	\begin{enumerate}
		\item \textit{Objects:} Pairs $(K, (n, m))$ where $n, m \ge 1$ and $K \subseteq \mathbb{R}^n \times \mathbb{R}^m$ is a compact, non-degenerate biaffinely convex set.
		\item \textit{Morphisms:} For objects $K \subseteq \mathbb{R}^{n_1+m_1}$ and $L \subseteq \mathbb{R}^{n_2+m_2}$, a morphism $\Theta \in \operatorname{Mor}_{\mathbf{BiAffConv}}(K, L)$ is a map $\Theta: K \to L$ given by a product of affine maps:
		\[
		\Theta(x, y) = (\phi(x), \psi(y)),
		\]
		where $\phi: \mathbb{R}^{n_1} \to \mathbb{R}^{n_2}$ and $\psi: \mathbb{R}^{m_1} \to \mathbb{R}^{m_2}$ are affine maps satisfying $\Theta(K) \subseteq L$.
	\end{enumerate}
	Composition of morphisms is given by standard function composition:
	\[
	(\phi_2, \psi_2) \circ (\phi_1, \psi_1) = (\phi_2 \circ \phi_1, \psi_2 \circ \psi_1),
	\]
	and the identity morphism on $K$ is $\operatorname{id}_K = (\operatorname{id}_{\mathbb{R}^n}, \operatorname{id}_{\mathbb{R}^m})|_K$.
\end{definition}

\begin{definition}
	The category of biaffine order unit spaces, denoted by $\mathbf{BiAOU}$, consists of:
	\begin{enumerate}
		\item \textit{Objects:} Finite-dimensional biaffine order unit spaces $(A, E_1, E_2, u_A, A^+)$ satisfying Definition~\ref{def:bimodule}.
		\item \textit{Morphisms:} For objects $B = (B, F_1, F_2, u_B, B^+)$ and $A = (A, E_1, E_2, u_A, A^+)$, a morphism $T \in \operatorname{Mor}_{\mathbf{BiAOU}}(B, A)$ is a unital, positive linear map $T: B \to A$ that respects the factor decomposition
		\[
		T = T_1 \otimes T_2,
		\]
		where $T_1: F_1 \to E_1$ and $T_2: F_2 \to E_2$ are unital linear maps on the factor AOU spaces ($T_1(u_{F_1}) = u_{E_1}$, $T_2(u_{F_2}) = u_{E_2}$, $T_1(F_1^+) \subseteq E_1^+$, $T_2(F_2^+) \subseteq E_2^+$) and $T(B^+) \subseteq A^+$.
	\end{enumerate}
	Composition is the standard tensor composition of linear maps $(S_1 \otimes S_2) \circ (T_1 \otimes T_2) = (S_1 \circ T_1) \otimes (S_2 \circ T_2)$, with identity $\operatorname{id}_A = \operatorname{id}_{E_1} \otimes \operatorname{id}_{E_2}$.
\end{definition}

\begin{remark}
	Note that if $T = T_1 \otimes T_2$ is positive and $T_1, T_2$ are unital, positivity of 
	the factor maps ($T_1(F_1^+) \subseteq E_1^+$ and $T_2(F_2^+) \subseteq E_2^+$) 
	is automatically satisfied. Indeed, for any $f \in F_1^+$, Axiom (3) applied to $B$ gives $f \otimes u_{F_2} \in B^+$. Because $T$ is positive and $T_2$ is unital, $T_1(f) \otimes u_{E_2} = T(f \otimes u_{F_2}) \in A^+$. Applying Axiom (3) to $A$ then immediately yields $T_1(f) \in E_1^+$.
\end{remark}

\begin{remark}
	In Definition~\ref{def:biaff_convex}, the objects are subsets of the Cartesian product 
	$\mathbb{R}^n \times \mathbb{R}^m$ where the two coordinate blocks are distinguished 
	and ordered (reflecting the fixed two-variable decomposition of functions $p(x,y)$). 
	Consequently, morphisms are chosen to be factor-preserving, mapping the first coordinate 
	block to the first and the second to the second. 
	
	Even when $n_1 = m_1 = n_2 = m_2$, coordinate permutations such as 
	$(x,y) \mapsto (y,x)$ (and their algebraic counterparts, flip isomorphisms 
	$a \otimes b \mapsto b \otimes a$) are deliberately excluded from the morphism class. 
	This reflects the convention in partial convexity that coordinate blocks represent 
	designated, distinct variables. 
\end{remark}

\subsection{The Duality Functors}

We now define the contravariant functors connecting the geometric and algebraic categories.

\begin{proposition}
	There is a functor 
	\[
	\mathcal{P}: \mathbf{BiAffConv}^{\mathrm{op}} \longrightarrow \mathbf{BiAOU}
	\]
	defined as follows:
	\begin{enumerate}
		\item \textit{On objects:} For $K \in \mathbf{BiAffConv}$, $\mathcal{P}(K) := \mathcal{P}_{\mathrm{biaff}}(K)$.
		\item \textit{On morphisms:} For $\Theta = (\phi, \psi) \in \operatorname{Mor}_{\mathbf{BiAffConv}}(K, L)$, the map $\mathcal{P}(\Theta) = \Theta^*: \mathcal{P}_{\mathrm{biaff}}(L) \to \mathcal{P}_{\mathrm{biaff}}(K)$ is defined by pullback
		\[
		\Theta^*(q) := q \circ \Theta \quad \text{for } q \in \mathcal{P}_{\mathrm{biaff}}(L).
		\]
	\end{enumerate}
\end{proposition}

\begin{proof}
	Let $K \subseteq \mathbb{R}^{n_1+m_1}$ and $L \subseteq \mathbb{R}^{n_2+m_2}$. By Proposition~\ref{prop:fnsp}, 
	$\mathcal{P}_{\mathrm{biaff}}(K)$ and $\mathcal{P}_{\mathrm{biaff}}(L)$ are biaffine order unit spaces 
	over $(\mathrm{Aff}(\pi_1(K)), \mathrm{Aff}(\pi_2(K)))$ and $(\mathrm{Aff}(\pi_1(L)), \mathrm{Aff}(\pi_2(L)))$, respectively.
	
	Let $\Theta = (\phi, \psi) \in \mathrm{Mor}_{\mathbf{BiAffConv}}(K, L)$. For any $f \in \mathrm{Aff}(\pi_1(L))$ 
	and $g \in \mathrm{Aff}(\pi_2(L))$, evaluating on $(x,y) \in K$ yields
	\[
	(f \otimes g)(\Theta(x,y)) = f(\phi(x)) g(\psi(y)) = (\phi^*f)(x) (\psi^*g)(y) = (\phi^*f \otimes \psi^*g)(x,y).
	\]
	Hence, $\Theta^* = \phi^* \otimes \psi^*$. Because $\phi: \mathbb{R}^{n_1} \to \mathbb{R}^{n_2}$ is an affine map 
	satisfying $\phi(\pi_1(K)) \subseteq \pi_1(L)$, the pullback $\phi^*(f) = f \circ \phi$ maps $\mathrm{Aff}(\pi_1(L))$ 
	linearly into $\mathrm{Aff}(\pi_1(K))$, and satisfies $\phi^*(1_{\pi_1(L)}) = 1_{\pi_1(K)}$. Similarly, 
	$\psi^*: \mathrm{Aff}(\pi_2(L)) \to \mathrm{Aff}(\pi_2(K))$ is a unital linear map.
	
	Furthermore, $\Theta^*(1_L) = 1_L \circ \Theta = 1_K$. If $q \in \mathcal{P}_{\mathrm{biaff}}(L)^+$, then 
	$q(w) \ge 0$ for all $w \in L$. Since $\Theta(K) \subseteq L$, we have 
	$\Theta^*(q)(x,y) = q(\Theta(x,y)) \ge 0$ for all $(x,y) \in K$, so $\Theta^*(q) \in \mathcal{P}_{\mathrm{biaff}}(K)^+$. 
	Thus $\Theta^*$ is a unital, positive linear map factoring as $\phi^* \otimes \psi^*$, confirming it is a valid 
	morphism in $\mathbf{BiAOU}$.
	
	Functoriality follows directly from function pullbacks: $\operatorname{id}_K^* = \operatorname{id}_{\mathcal{P}(K)}$ 
	and $(\Theta_2 \circ \Theta_1)^* = \Theta_1^* \circ \Theta_2^*$.
\end{proof}

\begin{proposition}\label{prop:f2}
	Let $\mathcal{C}$ be a choice of designated affine coordinate bases 
	$\mathcal{B}_A = (\{e_i\}_{i=0}^{n_A}, \{d_j\}_{j=0}^{m_A})$ for each object 
	$A \in \mathbf{BiAOU}$. Then there is a functor
	\[
	\mathcal{K}_{\mathcal{C}} : \mathbf{BiAOU}^{\mathrm{op}} \longrightarrow \mathbf{BiAffConv}
	\]
	defined by:
	\begin{enumerate}
		\item \textit{On objects:} For $A \in \mathbf{BiAOU}$, $\mathcal{K}_{\mathcal{C}}(A) := K_{A, \mathcal{B}_A} \subseteq \mathbb{R}^{n_A+m_A}$.
		\item \textit{On morphisms:} For $T = T_1 \otimes T_2 \in \mathrm{Mor}_{\mathbf{BiAOU}}(B, A)$, the map 
		$\mathcal{K}_{\mathcal{C}}(T) : K_{A, \mathcal{B}_A} \to K_{B, \mathcal{B}_B}$ is given by
		\[
		\mathcal{K}_{\mathcal{C}}(T)(x,y) := (\phi_T(x), \psi_T(y)),
		\]
		where $\phi_T(x) := (T_1^*(\phi_x)(f_k))_{k=1}^{n_B}$ and $\psi_T(y) := (T_2^*(\psi_y)(g_l))_{l=1}^{m_B}$.
	\end{enumerate}
	Moreover, any two choices of basis assignments yield naturally isomorphic functors.
\end{proposition}

\begin{proof}
	Let $A$ and $B$ be biaffine order unit spaces of types $(n_A, m_A)$ and $(n_B, m_B)$ over factor pairs $(E_1, E_2)$ and $(F_1, F_2)$, with designated bases $\mathcal{B}_A = (\{e_i\}_{i=0}^{n_A}, \{d_j\}_{j=0}^{m_A})$ and $\mathcal{B}_B = (\{f_k\}_{k=0}^{n_B}, \{g_l\}_{l=0}^{m_B})$ selected by $\mathcal{C}$. For convenience, write $K_A := K_{A, \mathcal{B}_A}$ and $K_B := K_{B, \mathcal{B}_B}$. By Proposition \ref{prop:sets}, $K_A$ and $K_B$ are compact non-degenerate biaffinely convex sets.   
	
	Let $T = T_1 \otimes T_2 \in \operatorname{Mor}_{\mathbf{BiAOU}}(B, A)$ and $\mathcal{K}_{\mathcal{C}}(T)(x, y) = (\phi_T(x), \psi_T(y))$. The coordinates of $\phi_T(x)$ are given by
	$$(\phi_T(x))_k = T_1^*(\phi_x)(f_k) = \phi_x(T_1(f_k))$$   For each $k \in \{1, \dots, n_B\}$, expanding $T_1(f_k) \in E_1$ with respect to the basis $\{e_i\}_{i=0}^{n_A}$ yields$$T_1(f_k) = c_{k0} u_{E_1} + \sum_{i=1}^{n_A} c_{ki} e_i$$   
	for coefficients $c_{k0}, c_{ki} \in \mathbb{R}$. Evaluating $\phi_x$ on $T_1(f_k)$ gives
	$$(\phi_T(x))_k = \phi_x(T_1(f_k)) = c_{k0} \phi_x(u_{E_1}) + \sum_{i=1}^{n_A} c_{ki} \phi_x(e_i) = c_{k0} + \sum_{i=1}^{n_A} c_{ki} x_i$$   
	Each component $(\phi_T(x))_k$ is an affine polynomial in $x = (x_1, \dots, x_{n_A})$, so $\phi_T: \mathbb{R}^{n_A} \to \mathbb{R}^{n_B}$ is an affine map. By the same reasoning, $\psi_T: \mathbb{R}^{m_A} \to \mathbb{R}^{m_B}$ is an affine map.   
	
	Next, we verify that $(\phi_T, \psi_T)(K_A) \subseteq K_B$. Let $(x, y) \in K_A$. By Definition \ref{def:statesp}, $\phi_x \otimes \psi_y \in S_{\mathrm{biaff}}(A) = S(A) \cap (S(E_1) \odot S(E_2))$. Under the adjoint $T^* = T_1^* \otimes T_2^*$, we have
	$$T^*(\phi_x \otimes \psi_y) = T_1^*(\phi_x) \otimes T_2^*(\psi_y)$$   Because $T_1: F_1 \to E_1$ is unital and positive,$$T_1^*(\phi_x)(u_{F_1}) = \phi_x(T_1(u_{F_1})) = \phi_x(u_{E_1}) = 1$$   
	and $T_1^*(\phi_x)(F_1^+) = \phi_x(T_1(F_1^+)) \subseteq \phi_x(E_1^+) \subseteq [0, \infty)$ since $\phi_x \in S(E_1)$. Thus $T_1^*(\phi_x) \in S(F_1)$, and similarly $T_2^*(\psi_y) \in S(F_2)$. Furthermore, since $T(B^+) \subseteq A^+$, for every $b \in B^+$ we have
	$$(T_1^*(\phi_x) \otimes T_2^*(\psi_y))(b) = (\phi_x \otimes \psi_y)(T(b)) \ge 0$$   
	Hence $T^*(\phi_x \otimes \psi_y) \in S(B) \cap (S(F_1) \odot S(F_2)) = S_{\mathrm{biaff}}(B)$. Applying the coordinate evaluation map $\iota_{\mathcal{B}_B}$ from Definition \ref{def:statesp} yields
	$$\iota_{\mathcal{B}_B}(T^*(\phi_x \otimes \psi_y)) = \left( (T_1^*(\phi_x)(f_k))_{k=1}^{n_B}, (T_2^*(\psi_y)(g_l))_{l=1}^{m_B} \right) = (\phi_T(x), \psi_T(y))$$   
	Since $T^*(\phi_x \otimes \psi_y) \in S_{\mathrm{biaff}}(B)$, its image under $\iota_{\mathcal{B}_B}$ lies in $K_B$. Thus $\mathcal{K}_{\mathcal{C}}(T)(K_A) \subseteq K_B$, confirming that $\mathcal{K}_{\mathcal{C}}(T)$ is a valid morphism in $\mathbf{BiAffConv}$.   
	
	Functoriality follows from the properties of linear adjoints: for $S \in \operatorname{Mor}_{\mathbf{BiAOU}}(C, B)$ and $T \in \operatorname{Mor}_{\mathbf{BiAOU}}(B, A)$, $(T \circ S)^* = S^* \circ T^*$. In coordinates, this yields $\phi_{T \circ S} = \phi_S \circ \phi_T$ and $\psi_{T \circ S} = \psi_S \circ \psi_T$, so
	$$\mathcal{K}_{\mathcal{C}}(T \circ S) = (\phi_S \circ \phi_T, \psi_S \circ \psi_T) = \mathcal{K}_{\mathcal{C}}(S) \circ \mathcal{K}_{\mathcal{C}}(T)$$   
	and $\mathcal{K}_{\mathcal{C}}(\operatorname{id}_A) = \operatorname{id}_{K_A}$.   
	
	Finally, let $\mathcal{C}$ and $\mathcal{C}'$ be two choices of designated affine coordinate bases, assigning bases $\mathcal{B}_A$ and $\mathcal{B}'_A$ to each $A \in \mathbf{BiAOU}$. By Lemma \ref{le:basis-change}, there is an isomorphism $\Theta_A = (\phi_A, \psi_A) : K_{A, \mathcal{B}_A} \to K_{A, \mathcal{B}'_A}$ in $\mathbf{BiAffConv}$ satisfying
	$$\Theta_A(\iota_{\mathcal{B}_A}(\omega)) = \iota_{\mathcal{B}'_A}(\omega)$$   
	for every rank-1 state $\omega \in S_{\mathrm{biaff}}(A)$. To show that the family $\{\Theta_A\}_{A \in \operatorname{Ob}(\mathbf{BiAOU})}$ defines a natural isomorphism $\Theta: \mathcal{K}_{\mathcal{C}} \Longrightarrow \mathcal{K}_{\mathcal{C}'}$, let $T \in \operatorname{Mor}_{\mathbf{BiAOU}}(B, A)$ and $(x, y) \in K_{A, \mathcal{B}_A}$. Setting $\omega = \phi_x \otimes \psi_y = \iota_{\mathcal{B}_A}^{-1}(x, y) \in S_{\mathrm{biaff}}(A)$, we compute
	$$(\Theta_B \circ \mathcal{K}_{\mathcal{C}}(T))(x, y) = \Theta_B(\iota_{\mathcal{B}_B}(T^*(\omega))) = \iota_{\mathcal{B}'_B}(T^*(\omega))$$   Conversely, since $\Theta_A(x, y) = \iota_{\mathcal{B}'_A}(\omega)$, the definition of $\mathcal{K}_{\mathcal{C}'}(T)$ gives$$(\mathcal{K}_{\mathcal{C}'}(T) \circ \Theta_A)(x, y) = \mathcal{K}_{\mathcal{C}'}(T)(\iota_{\mathcal{B}'_A}(\omega)) = \iota_{\mathcal{B}'_B}(T^*(\omega))$$   
	Thus $\Theta_B \circ \mathcal{K}_{\mathcal{C}}(T) = \mathcal{K}_{\mathcal{C}'}(T) \circ \Theta_A$, proving naturality.
\end{proof}

\begin{remark}
	The intrinsic geometric object associated with $A$ is the biaffine state space $S_{\mathrm{biaff}}(A)$, which is entirely coordinate-free. The embedding into Euclidean space $\mathbb{R}^{n+m}$ depends on selecting coordinate bases for $E_1$ and $E_2$. While $\mathcal{K}_{\mathcal{C}}$ serves as an explicit quasi-inverse once bases are selected, the categorical equivalence in the forthcoming Theorem \ref{thm:categorical_duality} does not rely on a canonical choice of $\mathcal{K}$, as $\mathcal{P}$ is shown directly to be fully faithful and essentially surjective.
\end{remark}

\subsection{Categorical Equivalence}

To prove that $\mathcal{P}$ is an equivalence of categories, we show that it is fully faithful.

\begin{proposition}\label{prop:faithful} 
	Let $K, L \in \mathbf{BiAffConv}$. The mapping 
	\[
	\operatorname{Mor}_{\mathbf{BiAffConv}}(K, L) \longrightarrow \operatorname{Mor}_{\mathbf{BiAOU}}(\mathcal{P}(L), \mathcal{P}(K)), \quad \Theta \longmapsto \Theta^*
	\]
	is injective.
\end{proposition}

\begin{proof}
	Let $\Theta_1 = (\phi_1, \psi_1)$ and $\Theta_2 = (\phi_2, \psi_2)$ be morphisms in $\operatorname{Mor}_{\mathbf{BiAffConv}}(K, L)$ such that $\Theta_1^* = \Theta_2^*$. 
	
	Let $x'_k \in \text{Aff}(\pi_1(L))$ denote the coordinate projection $x' \mapsto x'_k$ for $1 \le k \le n_2$. Then $x'_k \otimes 1_{\pi_2(L)} \in \mathcal{P}_{\mathrm{biaff}}(L)$. For every $(x, y) \in K$, we have
	\[
	\Theta_1^*(x'_k \otimes 1_{\pi_2(L)})(x, y) = (x'_k \otimes 1_{\pi_2(L)})(\phi_1(x), \psi_1(y)) = (\phi_1(x))_k.
	\]
	Since $\Theta_1^* = \Theta_2^*$, it follows that
	\[
	(\phi_1(x))_k = \Theta_2^*(x'_k \otimes 1_{\pi_2(L)})(x, y) = (\phi_2(x))_k \quad \text{for all } (x, y) \in K \text{ and } 1 \le k \le n_2.
	\]
	Consequently, $\phi_1(x) = \phi_2(x)$ for all $x \in \pi_1(K)$. 
	
	Because $K$ is non-degenerate of type $(n_1, m_1)$, Remark~\ref{rem:non-deg} ensures that $\pi_1(K)$ affinely spans $\mathbb{R}^{n_1}$. Since $\phi_1, \phi_2: \mathbb{R}^{n_1} \to \mathbb{R}^{n_2}$ are affine maps agreeing on an affine spanning set of $\mathbb{R}^{n_1}$, they must be identical: $\phi_1 = \phi_2$ on all of $\mathbb{R}^{n_1}$. 
	
	An identical argument applied to the coordinate projections $1_{\pi_1(L)} \otimes y'_l \in \mathcal{P}_{\mathrm{biaff}}(L)$ demonstrates that $\psi_1(y) = \psi_2(y)$ for all $y \in \pi_2(K)$, and by non-degeneracy of $\pi_2(K)$, $\psi_1 = \psi_2$ on all of $\mathbb{R}^{m_1}$. Hence $\Theta_1 = \Theta_2$, proving that $\mathcal{P}$ is faithful.
\end{proof}

\begin{proposition}\label{prop:full} 
	Let $K, L \in \mathbf{BiAffConv}$. The mapping 
	\[
	\operatorname{Mor}_{\mathbf{BiAffConv}}(K, L) \longrightarrow \operatorname{Mor}_{\mathbf{BiAOU}}(\mathcal{P}(L), \mathcal{P}(K)), \quad \Theta \longmapsto \Theta^*
	\]
	is surjective.
\end{proposition}

\begin{proof}
	Let $T \in \operatorname{Mor}_{\mathbf{BiAOU}}(\mathcal{P}(L), \mathcal{P}(K))$. By definition of morphisms in $\mathbf{BiAOU}$, $T$ factors as
	\[
	T = T_1 \otimes T_2,
	\]
	where $T_1: \text{Aff}(\pi_1(L)) \to \text{Aff}(\pi_1(K))$ and $T_2: \text{Aff}(\pi_2(L)) \to \text{Aff}(\pi_2(K))$ are unital positive linear maps.
	
	We construct a geometric morphism $\Theta = (\phi, \psi): K \to L$. Let $\{x'_k\}_{k=1}^{n_2}$ and $\{y'_l\}_{l=1}^{m_2}$ denote the canonical coordinate functions on $\mathbb{R}^{n_2}$ and $\mathbb{R}^{m_2}$ restricted to $\pi_1(L)$ and $\pi_2(L)$. For each $k$, $T_1(x'_k) \in \text{Aff}(\pi_1(K))$. Because $K$ is non-degenerate of type $(n_1, m_1)$, Remark~\ref{rem:non-deg} ensures that $\pi_1(K)$ affinely spans $\mathbb{R}^{n_1}$; hence $T_1(x'_k)$ extends uniquely to an affine function on the entire ambient space $\mathbb{R}^{n_1}$. We define $\phi: \mathbb{R}^{n_1} \to \mathbb{R}^{n_2}$ by
	\[
	\phi(x) := ((T_1(x'_1))(x), \dots, (T_1(x'_{n_2}))(x)).
	\]
	Since each component is an affine function on $\mathbb{R}^{n_1}$, $\phi$ is an affine map. By identical reasoning, defining $\psi: \mathbb{R}^{m_1} \to \mathbb{R}^{m_2}$ by
	\[
	\psi(y) := ((T_2(y'_1))(y), \dots, (T_2(y'_{m_2}))(y))
	\]
	yields an affine map. Set $\Theta(x, y) := (\phi(x), \psi(y))$.
	
	We claim that $\Theta(K) \subseteq L$. Let $(x, y) \in K$. Now $ev_{(x, y)} = ev_x \otimes ev_y \in S_{\mathrm{biaff}}(\mathcal{P}(K))$ and
	\[
	T^*(ev_{(x, y)}) = T_1^*(ev_x) \otimes T_2^*(ev_y).
	\]
	Because $T$ is a morphism in $\mathbf{BiAOU}$, Proposition~\ref{prop:f2} confirms that $T^*(ev_{(x, y)}) \in S_{\mathrm{biaff}}(\mathcal{P}(L))$. Evaluating the coordinate evaluation map $\iota_L$ on this state yields
	\[
	\iota_L(T^*(ev_{(x, y)})) = \left( (T_1^*(ev_x)(x'_k))_{k=1}^{n_2}, (T_2^*(ev_y)(y'_l))_{l=1}^{m_2} \right) = (\phi(x), \psi(y)) = \Theta(x, y).
	\]
	Consequently, $\Theta(x, y) \in \iota_L(S_{\mathrm{biaff}}(\mathcal{P}(L))) = K_{\mathcal{P}(L)}$. Because $L$ is biaffinely convex, Proposition~\ref{prop:setiso} ensures that $K_{\mathcal{P}(L)} = L$. Therefore, $\Theta(x, y) \in L$ for all $(x, y) \in K$, proving that $\Theta(K) \subseteq L$.
	
	Finally, we verify that $\Theta^* = T$. Any affine function $f \in \text{Aff}(\pi_1(L))$ expands uniquely as $f = c_0 1_{\pi_1(L)} + \sum_{k=1}^{n_2} c_k x'_k$. Since $T_1$ is linear and unital ($T_1(1_{\pi_1(L)}) = 1_{\pi_1(K)}$), evaluating $f$ at $\phi(x)$ gives
	\[
	f(\phi(x)) = c_0 + \sum_{k=1}^{n_2} c_k (\phi(x))_k = c_0 (T_1(1_{\pi_1(L)}))(x) + \sum_{k=1}^{n_2} c_k (T_1(x'_k))(x) = (T_1(f))(x).
	\]
	Similarly, $g(\psi(y)) = (T_2(g))(y)$ for every $g \in  \text{Aff}(\pi_2(L))$. Therefore, for any elementary tensor $f \otimes g \in  \text{Aff}(\pi_1(L)) \otimes  \text{Aff}(\pi_2(L))$ and any $(x, y) \in K$,
	\[
	\Theta^*(f \otimes g)(x, y) = f(\phi(x)) g(\psi(y)) = (T_1(f))(x) (T_2(g))(y) = (T(f \otimes g))(x, y).
	\]
	Hence $\Theta^*(f \otimes g) = T(f \otimes g)$ as functions on $K$. Because elementary tensors span $\mathcal{P}_{\mathrm{biaff}}(L)$, linearity ensures that $\Theta^* = T$. This completes the proof that $\mathcal{P}$ is full.
\end{proof}

\begin{theorem}
	\label{thm:categorical_duality}
	The functor 
	\[
	\mathcal{P}: \mathbf{BiAffConv}^{\mathrm{op}} \longrightarrow \mathbf{BiAOU}
	\]
	is an equivalence of categories. That is, $\mathbf{BiAffConv}$ is dually equivalent to $\mathbf{BiAOU}$.
\end{theorem}

\begin{proof}
	A functor is an equivalence if and only if it is faithful, full, and essentially surjective.
	Faithfulness and fullness of $\mathcal{P}$ were established in Proposition~\ref{prop:faithful} and Proposition~\ref{prop:full}, respectively.
	
	To establish essential surjectivity, let $A$ be an arbitrary object in $\mathbf{BiAOU}$ of type $(n,m)$. By selecting any affine coordinate bases for its factor spaces $E_1$ and $E_2$, Proposition \ref{prop:sets} guarantees that the coordinate biaffine state space $K_A \subseteq \mathbb{R}^{n+m}$ is an object in $\mathbf{BiAffConv}$. Furthermore, Proposition \ref{prop:fniso} provides an isomorphism $A \cong \mathcal{P}_{\mathrm{biaff}}(K_A) = \mathcal{P}(K_A)$ in $\mathbf{BiAOU}$. Thus, every object of $\mathbf{BiAOU}$ is isomorphic to the image under $\mathcal{P}$ of an object in $\mathbf{BiAffConv}$, which completes the proof.
\end{proof}


\end{document}